\documentclass[11pt]{article}
\usepackage{amsfonts, amsmath}
\usepackage{amssymb}
\usepackage{textcomp}
\usepackage{graphicx}
\usepackage{pdfsync}
\usepackage{color}
\usepackage{pbsi}
\def\sqr#1#2{{\vcenter{\vbox{\hrule height.#2pt
        \hbox{\vrule width.#2pt height#1pt \kern#2pt
        \vrule width.#2pt}
        \hrule height.#2pt}}}}

\newcommand{\nc}{\newcommand}
\nc{\parent}[1]{$[\![#1]\!]$}
\newtheorem{theorem}{Theorem}[section]
\newtheorem{lemma}{Lemma}[section]

\newtheorem{corollary}{Corollary}[section]
\newtheorem{proposition}{Proposition}[section]
\newtheorem{remark}{Remark}[section]

\newtheorem{assumption}{Assumption}[section]

\newenvironment{proof}{{\sc Proof.}\hspace{3mm}}{\qed}
\newenvironment{pf-main}{{\sc Proof of Theorem
    \ref{mainresult}.}\hspace{3mm}}{\qed}
\nc{\cadlag}{c\`{a}dl\`{a}g } \nc{\ba}{\begin{array}}
\nc{\ea}{\end{array}} \nc{\be}{\begin{equation}}
\nc{\ee}{\end{equation}} \nc{\bea}{\begin{eqnarray}}
\nc{\eea}{\end{eqnarray}} \nc{\bean}{\begin{eqnarray*}}
\nc{\eean}{\end{eqnarray*}} \nc{\bu}{\bullet} \nc{\nn}{\nonumber}
\nc{\cA}{{\mathcal A}} \nc{\cB}{{\mathcal B}} \nc{\cC}{{\mathcal C}}
\nc{\cD}{{\mathcal D}} \nc{\cN}{{\mathcal N}}\nc{\bbD}{\mathbb{D}}\nc{\bbH}{\mathbb{H}}
\nc{\bbF}{\mathbb{F}}\nc{\bbG}{\mathbb{G}}\nc{\cG}{{\mathcal G}} \nc{\cF}{{\mathcal F}}
\nc{\cS}{{\mathcal S}} \nc{\cZ}{{\mathcal Z}} \nc{\cU}{{\mathcal U}} \nc{\cH}{{\mathcal H}}
\nc{\cK}{{\mathcal K}} \nc{\cL}{{\mathcal L}} \nc{\cM}{{\mathcal M}}
\nc{\cO}{{\mathcal O}} \nc{\cP}{{\mathcal P}} \nc{\bbE}{\mathbb{E}}
\nc{\bbEQ}{\mathbb{E}_{\mathbb{Q}}} \nc{\eps}{\varepsilon}
\nc{\bbEP}{\mathbb{E}_{\mathbb{P}}}\nc{\bbL}{\mathbb{L}}
\nc{\bbP}{\mathbb{P}} \nc{\bbQ}{\mathbb{Q}} \nc{\del}{\partial}
\nc{\Om}{\Omega} \nc{\om}{\omega} \nc{\bbR}{\mathbb{R}}
\nc{\bbC}{\mathbb{C}} \nc{\bfr}{\begin{flushright}}
\nc{\efr}{\end{flushright}} \nc{\dXt}{\Delta X_{t}} \nc{\dXs}{\Delta
X_{s}} \nc{\bs}{\blacksquare} \nc{\dX}{\Delta X} \nc{\dY}{\Delta Y}
\nc{\dnkx}{\left(X(T^{n}_{k})-X(T^{n}_{k-1})\right)}
\nc{\esssup}{\mathrm{ess}\mbox{ }\mathrm{sup}}
\nc{\essinf}{\mathrm{ess}\mbox{ } \mathrm{inf}}
\nc{\dhats}{\widehat{\delta_s}} \nc{\tX}{\tilde{X}}
\nc{\tZ}{\tilde{Z}}\nc{\half} {\frac{1}{2}}

\def\rar{\rightarrow}

\nc{\qed}{\hfill $\blacksquare$}
\nc{\chf}{\mbox{$\mathbf1$}}

\begin{document}
\title{Explicit construction of a dynamic Bessel bridge of dimension $3$ }
\author{Luciano Campi\footnote{LAGA, University Paris 13, and CREST, campi@math.univ-paris13.fr.}\qquad Umut \c Cetin\footnote{Department of Statistics, London School of Economics, u.cetin@lse.ac.uk.} \qquad Albina Danilova\footnote{Department of Mathematics, London School of Economics, a.danilova@lse.ac.uk.}}
%\date{}
\maketitle

\begin{abstract}
Given a deterministically time-changed Brownian motion $Z$ starting
from $1$, whose time-change $V(t)$ satisfies $V(t) > t$ for all $t > 0$, we perform an explicit construction of a process $X$ which is
Brownian motion in its own filtration and that hits zero for the first
time at $V(\tau)$, where $\tau := \inf\{t>0: Z_t =0\}$. We also
provide the semimartingale decomposition of $X$ under the filtration
jointly generated by $X$ and $Z$. Our construction relies on a
combination of enlargement of filtration and filtering techniques. The
resulting process $X$ may be viewed as the analogue of a
$3$-dimensional Bessel bridge starting from $1$ at time $0$ and ending
at $0$ at the random time $V(\tau)$.  We call this {\em a dynamic
  Bessel bridge} since $V(\tau)$ is not known in advance.
  Our study is motivated by insider trading models with default risk, where the insider observes the firm's value continuously on time. The financial application, which uses results proved in the present paper, has been developed in the companion paper \cite{ccdfs}.
\end{abstract}

\section{Introduction}

In this paper, we are interested in constructing a Brownian motion starting from $1$ at time $t=0$ and conditioned to hit the level $0$ for the first time at a given random time. More precisely, let $Z$ be the deterministically time-changed Brownian motion $Z_t = 1+\int_0 ^t \sigma (s) dW_s $
and let $B$ be another standard Brownian motion independent of $W$. We denote $V(t)$ the associated time-change, i.e. $V(t) = \int_0 ^t \sigma^2 (s) ds$ for $t\geq 0$. Consider the first hitting time of $Z$ of the level $0$, denoted by $\tau$. Our aim is to build explicitly a process $X$ of the form $dX_t = dB_t + \alpha _t dt$, $X_0 =1$, where $\alpha$ is an integrable and adapted process for the filtration jointly generated by the pair $(Z,B)$ and satisfying the following two properties:
\begin{enumerate}
\item $X$ hits  level $0$ for the first time at time $V(\tau)$;
\item $X$ is a Brownian motion in its own filtration.
\end{enumerate}
This resulting process $X$ can be viewed as an analogue of
$3$-dimensional Bessel bridge with a random terminal time. Indeed, the
two properties above characterising $X$ can be reformulated as
follows: $X$ is a Brownian motion conditioned to hit $0$ for the first
time at the random time $V(\tau)$. In order to emphasise the distinct
property that $V(\tau)$ is not known at time $0$, we call this process
a {\em dynamic Bessel bridge of dimension 3}. The reason that $X$ hits $0$ at $V(\tau)$ rather than $\tau$ is simply due to the relationship between the first hitting times of $0$ by $Z$ and a standard Brownian motion starting at $1$.

The solution to the above problem consists of two parts with varying difficulties. The easy part is the construction of this process after time $\tau$. Since $V$ is a deterministic function, the first hitting time of $0$ is revealed at time $\tau$. Thus, one can use the well-known relationship between the 3-dimensional Bessel bridge and Brownian motion conditioned on its first hitting time to write for $ t \in (\tau, V(\tau))$
\[
dX_t=dB_t+\left\{\frac{1}{X_t}-\frac{X_t}{V(\tau)-t}\right\}\,dt.
\]
The difficult part is the construction of $X$ until time $\tau$. Thus, the challenge is to construct a Brownian motion which is conditioned to stay strictly positive until time $\tau$ using a drift term adapted to the filtration generated by $B$ and $Z$.

Our study is motivated by the equilibrium model with insider trading and default as in \cite{CaCe}, where a Kyle-Back type model with default is considered. In such a model, three agents act in the market of a defaultable bond issued by a firm, whose value process is modelled  as a Brownian motion and whose default time is set to be the first time that the firm's value hits a given constant default barrier. It has been shown in \cite{CaCe} that the equilibrium total demand for such a bond, after an appropriate translation, is a process $X^*$ which is a  $3$-dimensional Bessel bridge in insider's (enlarged) filtration but is a Brownian motion in its own filtration. These two properties
can be rephrased as follows: $X^*$ is a Brownian motion conditioned to
hit $0$ for the first time at the default time
$\tau$. However, the assumption that the insider knows the default time from the beginning may seem too strong from the modelling viewpoint. To approach the reality, one might consider a more realistic situation when the insider doesn't know the default time but however she can observe the evolution through time of the firm's value. Equilibrium considerations, akin to the ones employed in \cite{CaCe}, lead one to study the existence of processes which we called dynamic Bessel bridges of dimension 3 at the beginning of this introduction. The financial application announced here has been performed in the companion paper \cite{ccdfs}, where the tools developed in the present paper are used to solve explicitely the equilibrium model with default risk and dynamic insider information, as outlined above. We refer to that paper for further details.

We will observe in the next section that in order to make such a construction possible, one has to assume that $Z$ evolves faster than its underlying Brownian motion $W$, i.e. $V(t) \geq t$ for all $t \geq 0$. It can be proved (see next Section \ref{model}) that $V(t)$ cannot be equal to $t$ in any interval $(a,b)$ of $\bbR_+$. We will nevertheless impose a  stronger assumption that $V(t)>t$ for all $t>0$ in order to avoid unnecessary technicalities. In the context of the financial market described above this assumptions amounts to insider's information being more precise than that of the market maker (see \cite{BP} for a discussion of this assumption). Moreover, an additional assumption on the behaviour of the time change $V(t)$ in a neighbourhood of $0$ will be needed.

Apart from the financial application, which is our first motivation, such a problem is interesting from a probabilistic point of view as well. We have observed above that the difficult part in obtaining the dynamic Bessel bridge is the construction of a Brownian motion which is conditioned to stay strictly positive until time $\tau$ using a drift term adapted to the filtration generated by $B$ and $Z$. Such a construction is related to the conditioning of a Markov process, which has been the topic of various works in the literature. The canonical example of this phenomenon is the 3-dimensional Bessel process which is obtained when one conditions a standard Brownian motion to stay positive. Chaumont \cite{ConL} studies the analogous problem for L\'evy process whereas Bertoin and Doney \cite{BerD} are concerned with the situation for random walks and the convergence of their respective probability laws. Bertoin et al. \cite{BPC} constructs a Brownian path over a fixed time interval with a {\em given} minimum by performing  transformations on a Brownian bridge. More recently, Chaumont and Doney \cite{CD} revisits the L\'evy processes conditioned to stay positive and shows a Williams' type path decomposition result at the minimum of such processes. However, none of these approaches can be adopted to perform the construction that we are after since i) the time interval in which we condition the Brownian motion to be positive is random and not known in advance; and ii) we are not allowed to use transformations that are not adapted to the filtration generated by $B$ and $Z$.

The paper is structured as follows. In Section \ref{model}, we
formulate precisely our main result (Theorem \ref{th:main}) and
provide a partial justification for its assumptions. Section \ref{proof} contains the proof of Theorem \ref{th:main}, that uses, in particular, a technical result on the density of the signal process $Z$, whose proof is given in Section \ref{Zdensity}. Finally, several technical results used along our proofs have been relegated in the Appendix for reader's convenience.

\section{Formulation of the main result}
\label{model}
Let
$(\Omega, \cH ,\bbH=(\cH_t)_{t \geq 0}, \bbP)$ be a filtered
probability space satisfying the usual
conditions. We suppose that $\cH_0$ contains only the $\bbP$-null sets and { there exist two independent $\bbH$-Brownian motions, $B$ and $W$.}  We introduce the process
\be \label{sdeZ}
Z_t:=1+\int_0^t \sigma(s) dW_s,
\ee
for some $\sigma$ whose properties are given in the assumption below.
\begin{assumption}\label{a:sigma} There exist  a measurable function
 $\sigma : \bbR_+ \mapsto (0,\infty)$ such
 that:\begin{enumerate}
\item $V(t) :=  \int_0 ^t \sigma^2 (s)ds \in (t, \infty)$ for every $t> 0$;
\item There exists some $\eps>0$ such that $\int_0^{\eps}
  \frac{1}{(V(t)-t)^2}\, dt <\infty.$
\end{enumerate}
\end{assumption}
{ Notice that under this assumptions, $Z$ and $W$ generate the same minimal filtration satisfying the usual conditions.} Consider the following first hitting time of $Z$:
\be \label{d:tau}
\tau:=\inf\{t>0: Z_t=0\},
\ee
where $\inf \emptyset=\infty$ by convention. One can characterize the distribution of $\tau$
using the well-known distributions of first hitting times of a
standard Brownian motion. To this end let
\be \label{d:H}
H(t,a):=\bbP\left[T_a>t\right]=\int_t^{\infty} \ell(u,a)\, du,
\ee
for $a>0$ where
\bean
T_a&:=&\inf\{t>0: B_t=a\}, \mbox{ and}\\
\ell(t,a)&:=&\frac{a}{\sqrt{2 \pi t^3}}\exp\left(-\frac{a^2}{2
    t}\right).
\eean
Recall that
\[
\bbP[T_a>t|\cH_s]=\chf_{[T_a>s]} H(t-s, a- B_s), \qquad s <t.
\]
{ Thus, since $V$ is deterministic and strictly increasing, $(Z_{V^{-1}(t)})_{t \geq 0}$ is a standard Brownian motion in its own filtration starting at $1$, and consequently
%\be \label{tauVt} V(\tau) = V( \inf\{ t>0 : Z_t =0\}) = \inf \{ t>0 : Z_{V^{-1}(t)} = 0 \} \ee
\be \label{dist_tau}
\bbP[\tau>t|\cH_s]=\chf_{[\tau>s]} H(V(t)-V(s), Z_s).
\ee
Hence,
\[
  \bbP[V(\tau)>t]=H(t,1),
\]
 for every $t \geq 0$, i.e. $V(\tau)=T_1$ in distribution.} Here we would like to give another formulation for the function $H$ in
terms of the transition density of a {\em Brownian motion killed at
  0}. Recall that this transition density is given by
\be \label{d:q}
q(t,x,y):=\frac{1}{\sqrt{2 \pi t}}\left(\exp\left(-\frac{(x-y)^2}{2
        t}\right)-\exp\left(-\frac{(x+y)^2}{2 t}\right)\right),
\ee
for $x>0$ and $y>0$ (see Exercise (1.15), Chapter III in \cite{ry}). Then one has the identity
\be \label{H:alt}
H(t,a)=\int_0^{\infty} q(t,a,y)\, dy.
\ee

In the sequel, for any process $Y$, $\cF^Y$ is going to denote the minimal filtration satisfying the usual conditions and with respect to which $Y$ is adapted. The following  is the main result of this paper.

\begin{theorem} \label{th:main} There exists a unique strong solution to
\be
\mbox{\phantom{bos}} X_t = 1+ B_t+ \int_0^{\tau \wedge
  t}\frac{q_x(V(s)-s,X_s,Z_s)}{q(V(s)-s,X_s,Z_s)}\,ds +\int_{\tau \wedge
  t}^{V(\tau) \wedge
  t} \frac{\ell_a(V(\tau)-s,X_s)}{\ell(V(\tau)-s,X_s)}\,ds. \label{sdeX}
\ee
Moreover,
\begin{itemize}
\item[i)] Let $\cF^X_t=\cN \bigvee \sigma(X_s; s \leq t)$, where $\cN$ is the set of $\bbP$-null sets. Then,  $X$ is a standard Brownian motion with respect to $\bbF^X:=(\cF^X_t)_{t \geq 0}$;
\item[ii)] $V(\tau)=\inf\{t>0:X_t=0\}$.
\end{itemize}
\end{theorem}
The proof of this result is postponed to the next section. We conclude
this section by providing a justification for our assumption $V(t)>t$
for all $t >0$.

First, observe that we necessarily have $V(t)\geq t$ for any $t \geq 0$. This follows from the fact that if the construction in Theorem \ref{th:main} is possible, then $V(\tau)$ is an $\cF^{B,Z}$-stopping time since it is an exit time from the positive real line of the process $X$. Indeed, if $V(t)< t$ for some $t >0$ so that $V^{-1}(t)>t$, then $[V(\tau)<t]$ cannot belong to $\cF^{B,Z}_t$ since $[V(\tau)<t]\cap[\tau>t]=[\tau<V^{-1}(t)]\cap[\tau>t] \notin\cF^Z_t$, and that $\tau$ is not $\cF^B_{\infty}$-measurable.

{ We will next see that  when $V(t) \equiv t$ construction of a dynamic Bessel bridge is
not possible. Similar arguments will also show that $V(t)$ cannot be
equal to $t$ in an interval.} We are going to adapt to our setting the arguments used in \cite{fwy}, Proposition 5.1.

{ To this end consider any process $X_t = 1 + B_t + \int_0^t \alpha_s ds$ for some $\mathbb H$-adapted and integrable process $\alpha$. Assume that $X$ is a Brownian motion in its own filtration an that $\tau = \inf\{t: X_t =0\}$ a.s. and fix an arbitrary time $t \geq 0$.} The two processes $M^Z_s := \mathbb P [ \tau > t | \cF^Z _s]$ and $M_s ^X := \mathbb P[\tau >t | \cF^X _s]$, for $s\geq 0$, are uniformly integrable continuous martingales, the former for the filtration $\mathcal F^{Z,B}$ and the latter for the filtration $\mathcal F^X$.  In this case, Doob's optional sampling theorem can be applied to any pair of finite stopping times, e.g. $\tau \wedge s $ and $\tau$, to get the following:
\begin{eqnarray*} M^X _{\tau \wedge s} &=& \bbE[M^X _\tau |\mathcal F^X _{\tau \wedge s}] = \bbE[\mathbf 1_{\tau >t} |\mathcal F^X _{\tau \wedge s}] \\
&=& \bbE[M^Z _\tau |\mathcal F^X _{\tau \wedge s}] = \bbE[M^Z _{\tau \wedge s}|\mathcal F^X _{\tau \wedge s}],\end{eqnarray*}
where the last equality is just an application of the tower property of conditional expectations and the fact that $M^Z$ is martingale for the filtration $\mathcal F^{Z,B}$ which is bigger than $\cF^X$.
We also obtain
\[ \bbE[(M_{\tau \wedge s}^X - M_{\tau \wedge s}^Z)^2 ] =\bbE[(M^X _{\tau \wedge s})^2] + \bbE[(M^Z _{\tau \wedge s})^2] - 2 \bbE[M^X _{\tau \wedge s} M^Z _{\tau \wedge s}] .\]
{ Notice that, since the pairs $(X,\tau)$ and $(Z,\tau)$ have the same law by assumption, the random variables $M^X _{\tau \wedge s}$ and  $M^Z_{\tau \wedge s}$ have the same law too. This implies}
\[ \bbE[(M_{\tau \wedge s}^X - M_{\tau \wedge s}^Z)^2 ] =2\bbE[(M^X _{\tau \wedge s})^2] - 2 \bbE[M^X _{\tau \wedge s} M^Z _{\tau \wedge s}] .\]
On the other hand we can obtain
\[ \bbE[M^X _{\tau \wedge s} M^Z _{\tau \wedge s}] = \bbE[M^X _{\tau \wedge s} \bbE[ M^Z _{\tau \wedge s}|\mathcal F_{\tau \wedge s} ^X]] = \bbE[(M^X _{\tau \wedge s})^2],\]
which implies that $M^X _{\tau \wedge s} = M^Z _{\tau \wedge s}$ for all $s\geq 0$.
Using the fact that
\[ M_s ^Z = \mathbf 1_{\tau >s} H(t-s, Z_s),\quad M_s ^X = \mathbf 1_{\tau >s} H(t-s, X_s),\quad s<t,\] one has
\[ H(t-s, X_s ) = H(t-s,Z_s)\quad  \textrm{on } [\tau >s].\]
Since the function $a \mapsto H(u,a)$ is strictly monotone in $a$ whenever $u>0$, the last equality above implies that $X_s = Z_s$ for all $s < t$ on the set $[\tau >s]$. $t$ being arbitrary, we have that that $X^\tau _s = Z^\tau _s $ for all $s\geq 0$.

We have just proved that, before $\tau$, $X$ and $Z$ coincide, which contradicts the fact that $B$ and $Z$ are independent, so that the construction of a Brownian motion conditioned to hit $0$ for the first time at $\tau$ is impossible. A possible way out is to assume that the signal process $Z$ evolves faster than its underlying Brownian motion $W$, i.e. $V(t) \in (t,\infty)$ for all $t\geq 0$ as in our assumptions on $\sigma$. We prove our main result in the following section.

\section{Proof of the main result} \label{proof}

Note first that in order to show the existence and the uniqueness of the strong solution to the SDE
in (\ref{sdeX}) it suffices to show these properties for the following
SDE
\be \label{aux_sde}
Y_t = y+ B_t + \int_0^{\tau \wedge
  t}\frac{q_x(V(s)-s,Y_s,Z_s)}{q(V(s)-s,Y_s,Z_s)}\,ds, \qquad y > 0,
%Z_t &=& 1+ \int_0 ^t \sigma (s) dW_s , \label{aux_sdeZ}
\ee
and that $Y_{\tau} > 0$. Indeed, the drift term after $\tau$ is the same as that of a 3-dimensional Bessel bridge
from $X_{\tau}$ to $0$ over the interval $[\tau, V(\tau)]$. { Note that $V(\tau) = T_1$ in distribution implies that $\tau$ has the same law as $V^{-1}(T_1)$ which is finite since $T_1$ is finite and the function $V(t)$ is increasing to infinity as $t$ tends to infinity. Thus $\tau$ is a.s. finite.

By Corollary 5.3.23 in \cite{ks} the existence and  uniqueness of the strong solution of (\ref{aux_sde}) is equivalent to the existence of a weak solution and pathwise uniqueness of strong solution when the latter exists. More precisely, after proving pathwise uniqueness for the SDE (\ref{aux_sde}), and thus establishing the uniqueness of the system of  (\ref{sdeZ}) and (\ref{aux_sde}),  in Lemma \ref{uniqueness}, we will construct a weak solution, $(Y,Z)$ , to this system. The weak existence and pathwise uniqueness will then imply $(Y,Z)= h(1,y, \beta, W) $  for some measurable $h$ and some Brownian motion $\beta$ in view of Corollary 5.3.23 in \cite{ks}. Moreover, the second part of  Corollary 5.3.23 in \cite{ks} will finally give  us   $h(1, y, B, W)$ as the strong solution of the system described by (\ref{sdeZ}) and (\ref{aux_sde}).

In the sequel we will often work with a pair of  SDEs defining $(A,Z)$ where $A$ is a semimartingale given by an SDE whose drift coefficient depends on  $Z$.   In order to simplify the statements of the following results, we will  shortly  write existence and/or uniqueness of the  SDE for $A$, when we actually mean the corresponding property for the whole system.}

%More precisely, after proving pathwise uniqueness for the SDEs (\ref{aux_sde}) and (\ref{aux_sdeZ}) in Lemma \ref{uniqueness}, we will construct a weak solution to those SDEs
%\begin{eqnarray*} Z_t &=& 1+ \int_0 ^t \sigma (s) dW_s ,\\
%Y_t &=& y + B_t + \int_0 ^{t \wedge \tau} \frac{q_x}{q}(V(s) -s , Y_s ,Z_s) ds ,\quad y >0,\end{eqnarray*}
%using Brownian motions $(\beta , W)$ where $\beta$ is part of the solution, that is we exhibited a weak solution of the form $(Y, Z) = h(1, y, \beta, W)$ for some measurable function $h$. Since we will have pathwise uniqueness, we will apply the second part of Corollary 5.3.23 in [14] yielding that the process $h(1,y,B,W)$ for the same function $h$ with $B$ instead of $\beta$ is a strong solution of our SDEs on $[0,\tau]$. This will prove existence of a unique strong solution.

%To simplify the statement of the next results, we will make a slight abuse of terminology and we will refer only to the first SDE (\ref{aux_sde}) when talking about existence and uniqueness of the solution, even though we should more properly talk about the couple of SDEs.}

We start with demonstrating the pathwise uniqueness property.
\begin{lemma} \label{uniqueness} Pathwise uniqueness holds for the SDE in
  (\ref{aux_sde}).
\end{lemma}
\begin{proof}
It follows from direct calculations that
\be \label{qx/q}
\frac{q_x(t,x,z)}{q(t,x,z)}=\frac{z-x}{t} +
\frac{\exp\left(-\frac{2xz}{t}\right)}{1-\exp\left(-\frac{2xz}{t}\right)}\frac{2z}{t}.
\ee
Moreover, $\frac{q_x(t,x,z)}{q(t,x,z)}$ is decreasing in $x$ for fixed
$z$ and $t$. Now, suppose there exist two strong solutions, $Y^1$ and $Y^2$. Then
\[
(Y^1_{t \wedge \tau}-Y^2_{t \wedge \tau})^2 =2 \int_0^{\tau \wedge
  t}(Y^1_s-Y^2_s)\left\{\frac{q_x(V(s)-s,Y^1_s,Z_s)}{q(V(s)-s,Y^1_s,Z_s)}-\frac{q_x(V(s)-s,Y^2_s,Z_s)}{q(V(s)-s,Y^2_s,Z_s)}\right\}\,
  ds \leq 0.
\]
\end{proof}

The
existence of a weak solution will  be obtained in several steps. First
we show the existence of a weak solution to the SDE in the following
proposition and then conclude via Girsanov's theorem.
\begin{proposition} \label{p:sde2} There exists a unique strong solution to
\be \label{aux_sde2}
Y_t = y+B_t + \int_0^{\tau \wedge t} f(V(s)-s, Y_s, Z_s)\, ds \qquad y > 0,
\ee
where
\[
f(t,x,z):=\frac{\exp\left(-\frac{2xz}{t}\right)}{1-\exp\left(-\frac{2xz}{t}\right)}\frac{2z}{t}.
\]
Moreover, $\bbP[Y_{\tau}>0 \mbox{ and } Y_{t \wedge \tau} >0, \forall t>0]=1.$
\end{proposition}
\begin{proof}
Pathwise uniqueness can be shown as in Lemma \ref{uniqueness}; thus, its proof is omitted. Observe that if $Y$ is a solution to (\ref{aux_sde2}), then
\[
dY_t^2=2 Y_t dB_t +\left(2 \chf_{[\tau>t]} Y_t f (V(t)-t, Y_t, Z_t) + 1\right)\, dt.
\]
Inspired by this formulation we consider the following SDE:
\be \label{sde-U}
dU_t=2 \sqrt{|U_t|}dB_t + \left(2 \chf_{[\tau>t]} \sqrt{|U_t|} f(V(t)-t, \sqrt{|U_t|},
  Z_t) +1\right) \, dt,
\ee
with $U_0=y^2$. In Lemma \ref{aux_sde3} it is shown that there exists a weak solution to this SDE which is strictly positive in the interval $[0,\tau]$. This yields in particular that the absolute values can be removed from the SDE (\ref{sde-U}) considered over the interval $[0,\tau]$. Thus, it follows from an application of It\^{o}'s formula that $\sqrt{U}$ is a weak, therefore strong, solution to (\ref{aux_sde2}) in $[0,\tau]$ due to pathwise uniqueness and Corollary 5.3.23 in \cite{ks}. The global solution can now be easily constructed by the addition of $B_t- B_{\tau}$ after $\tau$. This further implies that $Y$ is strictly positive in $[0,\tau]$ since $\sqrt{U}$ is clearly strictly positive.
\end{proof}
\begin{lemma} \label{aux_sde3} There exists a weak solution to
\be \label{sdeU}
dU_t = 2 \sqrt{|U_t|}dB_t + \left(2 \sqrt{|U_t|} f(V(t)-t, \sqrt{|U_t|},
  Z_t) +1\right) \, dt,
\ee
with $U_0= y^2 $ upto and including $\tau$. Moreover, the solution is strictly positive in $[0, \tau]$.
\end{lemma}
\begin{proof}
Consider the measurable function $g: \bbR_+ \times \bbR^2 \mapsto [0,1]$ defined by
\[
g(t,x,z)=\left\{\ba{ll}
\sqrt{|x|} f(t,\sqrt{|x|},z), &\mbox{ for }(t,x,z) \in (0,\infty)\times \bbR \times \bbR_+ \\
1, & \mbox{ for } (t,x,z) \in (0,\infty)\times \bbR \times (-\infty,0)\\
 0, & \mbox{ for } (t,x,z) \in \{0\}\times \bbR^2
 \ea
 \right. ,
 \]
 and the following SDE:
 \be \label{sde-U1}
d\tilde{U}_t=2 \sqrt{|\tilde{U}_t|}dB_t + \left(2 g(V(t)-t, \sqrt{|\tilde{U}_t|},
  Z_t) +1\right) \, dt.
\ee

Observe that if we can show the existence of a positive weak solution to (\ref{sde-U1}), then $U =(\tilde{U}_{t\wedge \tau})_{t \geq 0}$ is a positive weak solution to (\ref{sdeU}) upto time $\tau$.

It follows
from Corollary 10.1.2 and Theorem 6.1.7 in \cite{SV} that the martingale problem defined by the stochastic differential equations for $(\tilde{U},Z)$ with the state space $\bbR^2$ is well-posed upto an explosion time, i.e. there exists a weak solution to (\ref{sde-U1}), along with (\ref{sdeZ}), valid upto the explosion time by Theorem 5.4.11 in \cite{ks}. Fix one of these solutions and call it $(\tilde{U}, Z)$. Then, since the range of $g$ is $[0,3]$, it follows from Lemma \ref{l:compz} that $\tilde{U}$ is nonnegative and there is no explosion.

Next it remains to show the strict positivity of $U$ in $[0, \tau]$. First, let $a$ and $b$ be strictly positive numbers such that
\[
\frac{a e^{-a}}{1-e^{-a}}=\frac{3}{4} \qquad \mbox{and} \qquad \frac{b e^{-b}}{1-e^{-b}}=\frac{1}{2}.
\]
As $\frac{x e^{-x}}{1-e^{-x}}$ is strictly decreasing for positive values of $x$, one has $0< a < b $. Now define the stopping time
\[
I_0:=\inf\{0 < t \leq \tau: \sqrt{U_t} Z_t \leq \frac{V(t)-t}{2} a\},
\]
where $\inf \emptyset = \tau$ by convention. As $\sqrt{U_{\tau}} Z_{\tau}=0$,  $\sqrt{U_0}
Z_0=y^2$, and $V(t)-t >0$ for $t>0$, we have that $0< I_0 < \tau$, $\nu^y$-a.s. by
continuity of $(U, Z)$ and $V$, where $\nu^y$ is the probability measure
associated to the fixed weak solution. Moreover, $U_t >0$ on the set $[t \leq I_0]$.

Note that $C_t:=\frac{2 \sqrt{U_t} Z_t}{V(t)-t}$ is continuous on $(0, \infty)$ and $C_{I_0}=a$.
Thus, $\bar{\tau}:=\inf\{t > I_0: C_t=0\}> I_0$. Consider the following sequence of stopping times:
\bean
J_n&:=&\inf\{I_n \leq t \leq \bar{\tau}: C_t \notin (0, b)\}\\
I_{n+1}&:=&\inf\{J_n \leq t \leq \bar{\tau}: C_t=a\}
\eean for $n \in \mathbb{N} \cup \{0\}$, where $\inf\emptyset =\bar{\tau}$ by convention.

Our aim is to show that  $\tau=\bar{\tau} = \lim_{n \rar \infty} J_n$, a.s.. We start with establishing the second equality. As $J_n$s are increasing and bounded by $\bar{\tau}$, the limit exists and is bounded by $\bar{\tau}$. Suppose that $J:=\lim_{n \rar \infty} J_n < \bar{\tau}$ with positive probability. Note that by construction we have $I_n \leq J_n \leq I_{n+1}$ and, therefore, $\lim_{n \rar \infty} I_n=J$. Since $C$ is continuous, one has $\lim_{n \rar \infty} C_{I_n}= \lim_{n \rar \infty} C_{J_n}$. However, as on the set $[J < \bar{\tau}]$ we have $C_{I_n}=a$ and $C_{J_n}=b$ for all $n$, we arrive at a contradiction. Therefore, $\bar{\tau}=J$.

Next, we will demonstrate that $\bar{\tau}=\tau$. Observe that since $\tau$ is finite, a.s., and $U$ does not explode until $\tau$, one has that $C_{\tau}=0$. Therefore, $\bar{\tau} \leq \tau$ and thus $C_{\bar{\tau}}=0$.  Suppose that $\bar{\tau} <\tau$ with positive probability. Then, we claim that on this set $C_{J_n}=b$ for all $n$, which will lead to a contradiction since then $b=\lim_{n \rar \infty} C_{J_n}=C_{\bar{\tau}}=0$. We will show our claim by induction.
\begin{enumerate}
\item For $n=0$, recall that $I_0 < \bar{\tau}$ by construction.  Also note that on $(I_0, J_0]$ the drift term in (\ref{sdeU}) is greater than
$2$ as $\frac{x e^{-x}}{1-e^{-x}}$ is strictly decreasing for positive values of $x$ and due to the choice of $a$ and $b$. Therefore the solution to (\ref{sdeU}) is strictly positive in
$(I_0, J_0]$ in view of Lemma \ref{l:compstop} since a 2-dimensional Bessel process is always strictly
positive. Thus,  $C_{J_0}=b$.
\item Suppose we have $C_{J_{n-1}}=b$. Then, due to continuity of $C$, $I_n < \bar{\tau}$. For the same reasons as before,  the solution to (\ref{sdeU}) is strictly positive in
$(I_n, J_n]$. Thus, $C_{J_n}=b$.
\end{enumerate}
Thus, we have shown that  for all $t >0$, $U_{\tau \wedge t} >0$, a.s.. In order to show that $U_{\tau}>0$
consider the stopping time $I:=\sup\{ I_n:I_n < \tau\}$. Then, we must have that $I <\tau$ a.s. since otherwise $a=C_I=C_{\tau}=0$, another contradiction. Similar to the earlier cases the drift term in $(I, \tau]$ is larger than $2$, thus, $U_{\tau}>0$ as well.
\end{proof}
\begin{proposition} \label{p:euy} There exists a unique strong solution to (\ref{aux_sde}) which is strictly positive on $[0,\tau]$.
\end{proposition}
\begin{proof}
Due to Proposition \ref{p:sde2} there exists a unique  strong solution, $Y$, of (\ref{aux_sde2}). Define $(L_t)_{t \geq 0}$ by $L_0=1$ and
\[
dL_t= \chf_{[\tau >t]} L_t \frac{Y_t-Z_t}{V(t)-t}\, dB_t.
\]
Observe that there exists a solution to the above equation since
\[
\int_0^t \chf_{[\tau>s]}\left(\frac{Y_s-Z_s}{V(s)-s}\right)^2ds < \infty, \mbox{ a.s. } \forall t\geq 0.
\]
Indeed, since $Y$ and $Z$ are well-defined and continuous upto $\tau$, we have  $\sup_{s \leq \tau} |Y_s-Z_s| < \infty$, a.s. and thus the above expression is finite in view of Assumption \ref{a:sigma}.2.

If $(L_t)_{t \geq 0}$ is a true martingale, then for any $T>0$,  $\bbQ^T$ on $\cH_T$ defined by
\[
\frac{d\bbQ^T}{d\bbP_T}=L_T,
\]
where $\bbP_T$ is the restriction of $\bbP$ to $\cH_T$, is a probability measure on $\cH_T$ equivalent to $\bbP_T$. Then, by Girsanov Theorem (see, e.g.,  Theorem 3.5.1 in \cite{ks}) under $\bbQ_T$
\[
Y_t=y+ \beta^T_t + \int_0^{\tau \wedge
  t}\frac{q_x(V(s)-s,Y_s,Z_s)}{q(V(s)-s,Y_s,Z_s)}\,ds,
\]
for $t \leq T$ where $\beta^T$ is a $\bbQ^T$-Brownian motion. Thus, $Y$ is a weak solution to (\ref{aux_sde}) on $[0, T]$. Therefore, due to Lemma \ref{uniqueness} and Corollary 5.3.23 in \cite{ks}, there exists a unique strong solution to (\ref{aux_sde}) on $[0,T]$, and it is strictly positive on $[0,\tau]$ since $Y$ has this property.  Since $T$ is arbitrary, this yields a unique strong solution on $[0, \infty)$ which is strictly positive on $[0,\tau]$.

Thus, it remains to show that $L$ is a true martingale. Fix $T>0$ and for some $0 \leq t_{n-1} <
t_n \leq T$  consider
\be \label{novikov}
\bbE\left[\exp\left(\frac{1}{2}\int_{t_{n-1} \wedge \tau}^{t_n \wedge \tau}
  \left(\frac{Y_t-Z_t}{V(t)-t}\right)^2 dt \right)\right].
\ee
As both $Y$ and $Z$ are positive until $\tau$, $(Y_t-Z_t)^2 \leq Y_t^2
+Z_t^2 \leq R_t +Z_t^2$ by comparison where $R$ satisfies
\[
R_t=y^2 + 2 \int_0^t \sqrt{R_s}dB_s + 3 t.
\]
Therefore, since $R$ and $Z$ are independent, the  expression in (\ref{novikov})
is bounded by
\bea
&&\bbE\left[\exp\left(\frac{1}{2}\int_{t_{n-1}}^{t_n}
  R_t\upsilon(t) dt \right)\right]
\bbE\left[\exp\left(\frac{1}{2}\int_{t_{n-1}}^{t_n}
  Z^2_t\upsilon(t) dt \right)\right]  \label{e:bnov1}
\\
&\leq& \bbE\left[\exp\left(\frac{1}{2}R^{\ast}_T\int_{t_{n-1}}^{t_n}
  \upsilon(t) dt \right)\right]
\bbE\left[\exp\left(\frac{1}{2}(Z^{\ast}_T)^2\int_{t_{n-1}}^{t_n}
  \upsilon(t) dt \right)\right],  \nn
  \eea
where $Y^{\ast}_t:=\sup_{s \leq t}|Y_s|$ for any \cadlag process
$Y$ and $\upsilon(t):=\left(\frac{1}{V(t)-t}\right)^2$.  Recall that $Z$ is only a time-changed
Brownian motion where the time change is deterministic and $R_t$ is
the square of the Euclidian norm of a 3-dimensional standard Brownian
motion with initial value $(y^2,0,0)$. Thus, since $V(T)>T$, the above expression is
going to be finite if
\be \label{e:bnov2}
E^{y \vee 1}\left[\exp\left(\frac{1}{2}(\beta^{\ast}_{V(T)})^2\int_{t_{n-1}}^{t_n}
  \upsilon(t) dt \right)\right] <\infty,
\ee
where $\beta$ is a standard Brownian motion and  $E^x$ is the
expectation with respect to the law of a standard Brownian  motion
starting at $x$. Indeed, it is clear that, by time change, (\ref{e:bnov2}) implies that the second expectation in the RHS of (\ref{e:bnov1}) is finite. Moreover, since $R_T ^*$ is the supremum over $[0,T]$ of a $3$-dimensional Bessel square process, it can be bounded above by the sum of three supremums of squared Brownian motions over $[0,V(T)]$ (remember that $V(T) >T$), which gives that (\ref{e:bnov2}) is an upper bound for the first expectation in the RHS of (\ref{e:bnov1}) as well.

In view of the reflection principle for standard
Brownian motion (see, e.g. Proposition 3.7 in Chap.~3 of \cite{ry}) the above expectation is going to be finite if
\be \label{cond}
\int_{t_{n-1}}^{t_n}
  \upsilon(t) dt < \frac{1}{V(T)}.
\ee

However, Assumption \ref{a:sigma} yields that $\int_0^T
 \upsilon(t) dt <\infty$. Therefore, we can find
 a finite sequence of real numbers $0=t_0<t_1<\ldots<t_{n(T)}=T$
 that satisfy (\ref{cond}). Since $T$ was arbitrary, this means that
 we can find a sequence $(t_n)_{n \geq 0}$ with $\lim_{n \rar
   \infty}t_n=\infty$ such that (\ref{novikov}) is finite for all
 $n$. Then, it follows from Corollary 3.5.14 in \cite{ks} that $L$ is
 a martingale.
\end{proof}

The above proposition establishes $0$ as a lower bound to the solution of (\ref{aux_sde}) over the interval $[0,\tau]$, however, one can obtain a tighter bound. Indeed, observe that $\frac{q_x}{q}(t,x,z)$ is strictly increasing in $z$ on $[0,\infty)$ for fixed $(t,x) \in \bbR_{++}^2$. Moreover,
\[
\frac{q_x}{q}(t,x,0):=\lim_{z \downarrow 0}\frac{q_x}{q}(t,x,z)=\frac{1}{x}-\frac{x}{t}.
\]
Therefore, $\frac{q_x}{q}(V(t)-t,Y_t,Z_t) > \frac{q_x}{q}(V(t)-t,Y_t,0)=\frac{1}{Y_t} - \frac{Y_t}{V(t)-t}$ for $t \in (0, \tau]$. Although $\frac{q_x}{q}(t,x,z)$ is not Lipschitz in $x$ (thus, standard comparison results don't apply), if $Y_0 < Z_0$ then the  comparison result of Exercise 5.2.19 in \cite{ks} can be applied to obtain $\bbP[Y_t \geq R_t; 0 \leq t <\tau]=1$ where $R$ is given by(\ref{e:tcb}).

However, this strict inequality may break down at $t=0$ when $Y_0 \geq Z_0$, and, thus, rendering the results of Exercise 5.2.19 is inapplicable.  Nevertheless, we will show in Proposition \ref{p:compare} that  $\bbP[Y_t \geq R_t; 0 \leq t <\tau]=1$ where $R$ is the solution of
\be \label{e:tcb}
R_t=y+ B_t + \int_0^t \left\{\frac{1}{R_s}- \frac{R_s}{V(s)-s}\right\}\, ds, \qquad y>0.
\ee
Before proving the comparison result we first establish that there exists a unique strong solution to the SDE above and it equals in law to a scaled, time-changed 3-dimensional Bessel process. We incidentally observe that the existence of a {\em weak} solution to an SDE similar to that in (\ref{e:tcb}) is proved in Proposition 5.1 in \cite{cl} along with its distributional properties. Unfortunately, our SDE (\ref{e:tcb}) cannot be reduced to theirs and moreover, in our setting, existence of a weak solution is not enough.

\begin{proposition} \label{p:tcb} There exists a unique strong solution to (\ref{e:tcb}). Moreover, the law of $R$ is equal to the law of $\tilde{R}=(\tilde{R}_t)_{t \geq 0}$, where $\tilde{R}_t=\lambda_t \rho_{\Lambda_t}$ where $\rho$ is a 3-dimensional Bessel process starting at $y$ and
\bean
\lambda_t&:=&\exp\left(-\int_0^t \frac{1}{V(s)-s}\, ds\right), \\
\Lambda_t&:=&\int_0^t \frac{1}{\lambda^2_s}\,ds.
\eean
\end{proposition}
\begin{proof} Note that $\frac{1}{x}-\frac{x}{t}$ is decreasing in $x$ and, thus, pathwise uniqueness holds for (\ref{e:tcb}). Thus, it suffices to find a weak solution for the existence and the uniqueness of strong solution. Consider the 3-dimensional Bessel process $\rho$ which is the unique strong solution (see Proposition 3.3 in Chap.~VI in \cite{ry}) to
\[
\rho_t=y + B_t + \int_0^t \frac{1}{\rho_s}\, ds.
\]
Therefore, $\rho_{\Lambda_t}= y + B_{\Lambda_t} + \int_0^{\Lambda_t}\frac{1}{\rho_s}\, ds.$ Now, $M_t=B_{\Lambda_t}$ is a martingale with respect to the time-changed filtration $(\cH_{\Lambda_t})$ with quadratic variation given by  $\Lambda$. By  integration by parts we see that
\[
d(\lambda_t \rho_{\Lambda_t})=\lambda_t dM_t + \left\{\frac{1}{\lambda_t \rho_{\Lambda_t}} - \frac{\lambda_t \rho_{\Lambda_t}}{V(t)-t}\right\}dt.
\]
Since $\lambda_0 \rho_{\Lambda_0}=y$ and $\int_0^t \lambda^2_s d[M,M]_s=t$, we see that $\lambda_t \rho_{\Lambda_t}$ is a weak solution to (\ref{e:tcb}). This obviously implies the equivalence in law.
\end{proof}

\begin{proposition} \label{p:compare} Let $R$ be the unique strong solution to (\ref{e:tcb}). Then, $\bbP[Y_t \geq R_t; 0 \leq t <\tau]=1$ where $Y$ is the unique strong solution of (\ref{aux_sde}).
\end{proposition}
\begin{proof} Note that
\[
R_t - Y_t = \int_0^t \left\{\frac{q_x}{q}(V(s)-s,R_s, 0)-\frac{q_x}{q}(V(s)-s,Y_s, Z_s)\right\}ds,\]
so that by Tanaka's formula  (see Theorem 1.2 in Chap.~VI of \cite{ry})
\bean
(R_t-Y_t)^+&=& \int_0^t \chf_{[R_s>Y_s]}\left\{\frac{q_x}{q}(V(s)-s,R_s, 0)-\frac{q_x}{q}(V(s)-s,Y_s, Z_s)\right\}ds\\
&=&\int_0^t \chf_{[R_s>Y_s]}\left\{\frac{q_x}{q}(V(s)-s,R_s, 0)-\frac{q_x}{q}(V(s)-s,Y_s, 0)\right\}ds \\
&&+\int_0^t \chf_{[R_s>Y_s]}\left\{\frac{q_x}{q}(V(s)-s,Y_s, 0)-\frac{q_x}{q}(V(s)-s,Y_s, Z_s)\right\}ds\\
&\leq&\int_0^t \chf_{[R_s>Y_s]}\left\{\frac{q_x}{q}(V(s)-s,R_s, 0)-\frac{q_x}{q}(V(s)-s,Y_s, 0)\right\}ds,
\eean
since the local time of $R-Y$ at $0$ is identically $0$ (see Corollary 1.9 n Chap.~VI of \cite{ry}). Let $\tau_n:=\inf\{t>0: R_t \wedge Y_t= \frac{1}{n}\}.$ Note that as $R$ is strictly positive and $Y$ is strictly positive on $[0, \tau]$, $\lim_{n \rar \infty} \tau_n >\tau$. Since for each $t \geq 0$
\[
\left|\frac{q_x}{q}(t,x, 0)-\frac{q_x}{q}(t,y, 0)\right| \leq \left(\frac{1}{t} + \frac{1}{n^2}\right)|x-y|
 \]
 for all $x, y \in [1/n, \infty)$, we have
 \[
 (R_{t\wedge \tau_n} -Y_{t \wedge \tau_n})^+ \leq \int_0^{t} (R_{s\wedge \tau_n}-Y_{s\wedge \tau_n})^+\left(\frac{1}{V(s)-s} + \frac{1}{n^2}\right)ds.
 \]
 Thus, by Gronwall's inequality (see Exercise 14 in Chap.~V of \cite{Pr}), we have $(R_{t\wedge \tau_n} -Y_{t \wedge \tau_n})^+ = 0$ since
 \[
 \int_0^t \left(\frac{1}{V(s)-s} + \frac{1}{n^2}\right)ds < \infty
 \]
 by Assumption \ref{a:sigma}. Thus, the claim follows from the continuity of $Y$ and $R$ and the fact that $\lim_{n \rar \infty} \tau_n >\tau$.
\end{proof}

\begin{remark} \label{r:compare} Note that the above proof does not use the particular SDE satisfied by $Z$. The result of the above proposition will remain valid as long as $Z$ is nonnegative and $Y$ is the unique strong solution of (\ref{aux_sde}), strictly positive on $[0, \tau]$.
\end{remark}
Since the solution to (\ref{aux_sde}) is strictly positive on $[0,\tau]$ and the drift term in (\ref{sdeX}) after $\tau$ is the same as that of a 3-dimensional Bessel bridge from $X_{\tau}$ to $0$ over $[\tau, V(\tau)]$, we have proved
\begin{proposition}\label{p:up} There exists a unique strong solution to (\ref{sdeX}). Moreover, the solution is strictly positive in $[0, \tau]$.
\end{proposition}
{ Using the well-known properties of a 3-dimensional Bessel
  bridge (see, e.g., Section 12.1.3, in particular expression (12.9)
  in \cite{abm}), we also have the following}
\begin{corollary}
Let $X$ be the unique strong solution of (\ref{sdeX}). Then,
\[
V(\tau)=\inf\{t>0: X_t=0\}.
\]
\end{corollary}

Thus, in order to finish the proof of Theorem \ref{th:main} it remains to show that $X$ is a standard Brownian motion in its own filtration. We will achieve this result in several steps. First, we will obtain the canonical decomposition of $X$ with respect to the minimal filtration,  $\bbG$,  satisfying the usual conditions such that $X$ is $\bbG$-adapted and $\tau$ is a $\bbG$-stopping time. More precisely, $\bbG=(\cG_t)_{t \geq 0}$ where $\cG_t= \cap_{u >t} \tilde{\cG}_u$, with
$\tilde{\cG}_t:=\mathcal{N} \bigvee \sigma(\{X_s, s \leq t\}, \tau \wedge t)$ and
$\mathcal{N}$ being the set of $\bbP$-null sets. Then, we will initially enlarge this filtration with $\tau$ to show that the canonical decomposition of $X$ in this filtration is the same as that of a Brownian motion starting at $1$ in its own filtration enlarged with its first hitting time of $0$. This observation will allow us to conclude that the law of $X$ is the law of a Brownian motion.

In order to carry out this procedure we will use the following key result, the proof of which is deferred until the next section for the clarity of the exposition. We recall that
\[ H(t,a) = \int_0 ^{\infty} q(t,a,y) dy ,\]
where $q(t,a,y)$ is the transition density of a Brownian motion killed at $0$.
\begin{proposition} \label{c_density} Let $X$ be the unique strong
  solution of (\ref{sdeX}) and $f: \bbR_+ \mapsto \bbR$ be a bounded measurable
  function with a compact support contained in $(0,\infty)$. Then
\[
\bbE[\chf_{[\tau>t]}f(Z_t)|\cG_t]=\chf_{[\tau>t]}\int_0^{\infty}f(z)
\frac{q(V(t)-t,X_t,z)}{H(V(t)-t, X_t)}\,dz.
\]
\end{proposition}
Using the above proposition we can easily obtain the $\bbG$-canonical
decomposition of $X$.
\begin{corollary}Let $X$ be the unique strong
  solution of (\ref{sdeX}). Then,
\[
M_t:=X_t -1- \int_0^{\tau \wedge
  t}\frac{H_x(V(s)-s,X_s)}{H(V(s)-s,X_s)}\,ds -\int_{\tau \wedge
  t}^{V(\tau) \wedge
  t} \frac{\ell_a(V(\tau)-s,X_s)}{\ell(V(\tau)-s,X_s)}\,ds
\]
is a standard $\bbG$-Brownian motion starting at $0$.
\end{corollary}
\begin{proof}
It follows from  Theorem 8.1.5 in \cite{ka} and Lemma \ref{l:2moment}
that
\[
X_t -1- \int_0^{t}\bbE\left[\chf_{[\tau>s]}\frac{q_x(V(s)-s,X_s,Z_s)}{q(V(s)-s,X_s,Z_s)}\bigg|\cG_s\right]\,ds -\int_{\tau \wedge
  t}^{V(\tau) \wedge
  t} \frac{\ell_a(V(\tau)-s,X_s)}{\ell(V(\tau)-s,X_s)}\,ds
\]
is a $\bbG$-Brownian motion. However,
\bean
&&\bbE\left[\chf_{[\tau>s]}\frac{q_x(V(s)-s,X_s,Z_s)}{q(V(s)-s,X_s,Z_s)}\bigg|\cG_s\right]\\
&=&\chf_{[\tau>s]}\int_0^{\infty}\frac{q_x(V(s)-s,X_s,z)}{q(V(s)-s,X_s,z)}\frac{q(V(s)-s,X_s,z)}{H(V(s)-s,
    X_s)}\,dz \\
&=&\chf_{[\tau>s]}\frac{1}{H(V(s)-s, X_s)}\int_0^{\infty}q_x(V(s)-s,X_s,z)\,
dz \\
&=&\chf_{[\tau>s]}\frac{1}{H(V(s)-s, X_s)}\frac{\partial}{\partial
  x}\int_0^{\infty}q(V(s)-s,x,z)\,dz\bigg|_{x=X_s} \\
&=&\chf_{[\tau>s]}\frac{H_x(V(s)-s, X_s)}{H(V(s)-s, X_s)}.
\eean
\end{proof}

A naive way to show that $X$ as a solution of (\ref{sdeX}) is a Brownian
motion is to calculate the conditional distribution of $\tau$ given
the minimal filtration generated by $X$ satisfying the usual
conditions. Although, as we will see later, the conditional distribution of $V(\tau)$ given an observation of $X$ is defined by the function $H$ as defined in (\ref{d:H}), verification of this fact leads to a highly non-standard filtering problem. For this reason we use an alternative approach which utilizes the well-known decomposition of Brownian motion conditioned on its first hitting time as in \cite{CaCe}.

We shall next find the
canonical decomposition of $X$ under $\bbG^{\tau}:=(\cG^{\tau}_t)_{t
  \geq 0}$ where $\cG^{\tau}_t=\cG_t \bigvee \sigma(\tau)$. Note that
$\cG^{\tau}_t=\cF^X_{t+}  \bigvee \sigma(\tau)$. Therefore, the canonical
decomposition of $X$ under $\bbG^{\tau}$ would be its canonical
decomposition with respect to its own filtration initially enlarged
with $\tau$. As we shall see in the next proposition it will be the
same as the canonical decomposition of a Brownian motion in its own
filtration initially enlarged with its first hitting time of $0$.
\begin{proposition} \label{p:ie} Let $X$ be the unique strong
  solution of (\ref{sdeX}). Then,
\[
X_t-1 - \int_{0}^{V(\tau) \wedge
  t} \frac{\ell_a(V(\tau)-s,X_s)}{\ell(V(\tau)-s,X_s)}\,ds
\]
is a standard $\bbG^{\tau}$-Brownian motion starting at $0$.
\end{proposition}
\begin{proof} First, we will
  determine the law of $\tau$ conditional on $\cG_t$ for each $t$.
Let $f$ be a test function. Then
\bean
&&\bbE\left[\chf_{[\tau>t]}f(\tau)|\cG_t\right]=\bbE\left[\bbE\left[\chf_{[\tau>t]}f(\tau)|\cH_t\right]\bigg|\cG_t\right]\\
&=&\bbE\left[\chf_{[\tau>t]}\int_{t}^{\infty}f(u)\sigma^2(u)\ell(V(u)-V(t),Z_t)\,du\bigg|\cG_t\right]
\\
&=&\chf_{[\tau>t]}\int_{t}^{\infty}f(u)\sigma^2(u)\int_0^{\infty}\ell(V(u)-V(t),z)
\frac{q(V(t)-t,X_t,z)}{H(V(t)-t, X_t)}\,dz\,du\\
&=&-\chf_{[\tau>t]}\int_t^{\infty}f(u)\sigma^2(u)\int_0^{\infty}H_t(V(u)-V(t),z)
\frac{q(V(t)-t,X_t,z)}{H(V(t)-t, X_t)}\,dz\,du\\
&=&-\chf_{[\tau>t]}\int_t^{\infty}f(u)\sigma^2(u)
\frac{\partial}{\partial
  s}\int_0^{\infty}\int_0^{\infty}q(s,z,y)\,dy\,\frac{q(V(t)-t,X_t,z)}{H(V(t)-t,
  X_t)}\,dz\bigg|_{s=V(u)-V(t)}\,du \\
&=&-\chf_{[\tau>t]}\int_t^{\infty}f(u)\sigma^2(u)
\frac{\partial}{\partial
  s}\int_0^{\infty}\int_0^{\infty} \frac{q(V(t)-t,X_t,z)}{H(V(t)-t,
  X_t)} q(s,z,y) \,dz\,dy\bigg|_{s=V(u)-V(t)}\,du \\
&=&-\chf_{[\tau>t]}\int_t^{\infty}f(u)\sigma^2(u)
\frac{\partial}{\partial
  s}\int_0^{\infty} \frac{q(V(t)-t+s,X_t,y)}{H(V(t)-t,
  X_t)} \,dy \bigg|_{s=V(u)-V(t)}\,du\\
&=&-\chf_{[\tau>t]}\int_t^{\infty}f(u)\sigma^2(u)
 \frac{H_t(V(u)-t,X_t)}{H(V(t)-t,
  X_t)} \,du\\
&=&\chf_{[\tau>t]}\int_t^{\infty}f(u)\sigma^2(u)
 \frac{\ell(V(u)-t,X_t)}{H(V(t)-t,
  X_t)} \,du.
\eean
Thus, $\bbP[\tau \in du, \tau>t|\cG_t]=\chf_{[\tau>t]}\sigma^2(u)
 \frac{\ell(V(u)-t,X_t)}{H(V(t)-t,
  X_t)} \,du.$

Then, it follows from Theorem 1.6 in \cite{my} that
\[
M_t - \int_0^{\tau \wedge t}\left(\frac{\ell_a(V(\tau)-s,
    X_s)}{\ell(V(\tau)-s,X_s)}
  -\frac{H_x(V(s)-s,X_s)}{H(V(s)-s,X_s)}\right)\,ds
\]
is a $\bbG^{\tau}$-Brownian motion as in Example 1.6 in \cite{my}. This completes the proof.\end{proof}
\begin{corollary} Let $X$ be the unique strong solution of (\ref{sdeX}). Then, $X$ is a Brownian motion with respect to $\bbF^X$.
\end{corollary}
\begin{proof}
It follows from Proposition \ref{p:ie} that  $\bbG^{\tau}$-
decomposition of $X$ is given by
\[
X_t=1 + \mu_t + \int_0^{V(\tau) \wedge t}
\left\{\frac{1}{X_s}-\frac{X_s}{V(\tau)-s}\right\}\,ds,
\]
where $\mu$ is a standard $\bbG^{\tau}$-Brownian motion vanishing at $0$.
Thus, $X$ is a 3-dimensional Bessel bridge from $1$ to $0$ of length
$V(\tau)$. As $V(\tau)$ is the first hitting time of $0$ for $X$ and $V(\tau)=T_1$ in distribution, the result follows using
the same argument as in Theorem 3.6 in \cite{CaCe}.
\end{proof}

Next section is devoted to the proof of Proposition \ref{c_density}.

\section{Conditional density of $Z$}\label{Zdensity}
Recall from Proposition \ref{c_density} that we are interested in the conditional distribution of $Z_t$ on the set $[\tau>t]$. To this end we introduce the following change of measure on $\cH_t$. Let $\bbP_t$ be the restriction of $\bbP$ to $\cH_t$ and define $\bbP^{\tau, t}$ on $\cH_t$ by
\[
\frac{d\bbP^{\tau, t}}{d\bbP_t}=\frac{\chf_{[\tau>t]}}{\bbP[\tau>t]}.
\]

Note that this measure change is equivalent to an {\em h-transform} on the paths of $Z$ until time $t$ where the h-transform is defined by the function $H(V(t)-V(\cdot), \cdot)$ and $H$ is the function defined in (\ref{d:H}) (see Part 2, Sect.~VI.13 of \cite{D} for the definition and properties of h-transforms). Note also that $(\chf_{[\tau>s]}H(V(t)-V(s), Z_s))_{s \in [0,t]}$ is a $(\bbP, \bbH)$-martingale as a consequence of (\ref{dist_tau}). Therefore, an application of Girsanov's theorem yields that under $\bbP^{\tau, t}$ $(X,Z)$ satisfy
\bea
dZ_s&=&\sigma(s)d\beta^t_s + \sigma^2(s)\frac{H_x(V(t)-V(s), Z_s)}{H(V(t)-V(s), Z_s)}ds \label{e:z4f} \\
dX_s&=&dB_s + \frac{q_x(V(s)-s,X_s,Z_s)}{q(V(s)-s,X_s,Z_s)}\,ds, \label{e:x4f}
\eea
with $X_0=Z_0=1$ and $\beta^t$ being a $\bbP^{\tau, t}$-Brownian motion. Moreover, due to the property of h-transforms,  transition density of $Z$ under $\bbP^{\tau, t}$ is given by
\be \label{e:tdz}
\bbP^{\tau, t}[Z_s\in dz|Z_r=x]=q(V(s)-V(r), x, z) \frac{H(V(t)-V(s), z)}{H(V(t)-V(r), x)}.
\ee
Thus, $\bbP^{\tau, t}[Z_s\in dz|Z_r=x]=p(V(t); V(r), V(s), x,z)$
where
\be \label{e:td-bm}
p(t;r ,s, x, z)=q(s-r, x,z)\frac{H(t-s, z)}{H(t-r,x)}.
\ee
Note that $p$ is the transition density of the Brownian motion killed at $0$ after the analogous h-transform where the h-function is given by $H(t-s, x)$.

\begin{lemma}\label{l:rcf} Let $\cF^{\tau,t, X}_s=\sigma(X_r; r \leq s) \vee \cN^{\tau,t}$ where $X$ is the process defined by (\ref{e:x4f}) with $X_0=1$, and $\cN^{\tau, t}$ is the collection of $\bbP^{\tau,t}$-null sets. Then the  filtration $(\cF^{\tau,t,X}_s)_{s \in [0,t]}$ is right-continuous.
\end{lemma}
The proof of the above lemma is trivial once we observe that $(\cF^{\tau,t, X}_{\tau_n \wedge s})_{s \in [0,t]}$, where $\tau_n:=\inf\{s>0:X_s=\frac{1}{n}\}$, is right continuous. This follows from the observation that $X^{\tau_n}$ is a Brownian motion under an equivalent probability measure, which can be shown using the arguments of Proposition \ref{p:euy} along with the identity (\ref{qx/q}) and the fact that $\frac{1}{X}$ is bounded upto $\tau_n$. Thus, for each $n$ one has
\bean
\cF^{\tau,t, X}_{\tau_n }\cap\cF^{\tau,t, X}_{u}&=&\cF^{\tau,t, X}_{\tau_n \wedge u}=\bigcap_{s>u}\cF^{\tau,t, X}_{\tau_n \wedge s}\\
&=&\bigcap_{s>u}\left(\cF^{\tau,t, X}_{\tau_n}\cap\cF^{\tau,t, X}_{s}\right)=\left(\bigcap_{s >u}\cF^{\tau,t, X}_{s}\right)\cap\cF^{\tau,t, X}_{\tau_n}
\eean
Indeed, since $\cup_{n}\cF^{\tau,t, X}_{\tau_n}=\cF^{\tau,t, X}_{\tau}$, letting $n$ tend to infinity yields the conclusion.

%to show the last equality take A in Ftau and observe that A=\cup A_n \cup null set where A_n=A \cap [tau_n>t], which %belongs to Ftau_n. see also lemma 7.1 in kallenberg.

The reason for the introduction of the probability measure $\bbP^{\tau, t}$ and the filtration  $(\cF^{\tau,t,X}_s)_{s \in [0,t]}$ is that $(\bbP^{\tau, t},(\cF^{\tau,t,X}_s)_{s \in [0,t]})$-conditional distribution of $Z$ can be characterised by a {\em Kushner-Stratonovich equation} which is well-defined. Moreover, it gives us $(\bbP,\cG)$-conditional distribution of $Z$.  Indeed, observe that $\bbP^{\tau, t}[\tau>t]=1$ and  for any set $E\in \cG_t$, $\chf_{[\tau>t]} \chf_{E}=\chf_{[\tau>t]}\chf_{F}$ for some set $F \in \cF^{\tau,t,X}_t$ (see Lemma 5.1.1 in \cite{br} and the remarks that follow). Then, it follows from the definition of conditional expectation that
\be \label{e:eq}
\bbE\left[f(Z_t)\chf_{[\tau>t]}|\cG_t\right]=\chf_{[\tau>t]}\bbE^{\tau,t}\left[f(Z_t)\big|\cF^{\tau,t,X}_t\right], \bbP-a.s..
\ee
Thus, it is enough to compute the conditional distribution of $Z$ under $\bbP^{\tau,t}$ with respect to $(\cF^{\tau,t,X}_s)_{s \in [0,t]}$. In order to achieve this goal we will use the characterization of the conditional distributions obtained by Kurtz and Ocone \cite{ko}. We refer the reader to \cite{ko} for all unexplained details and terminology.

Let $\cP$ be the set of probability measures on the Borel sets of $\bbR_+$ topologized by weak convergence.  Given $m \in \cP$ and $m-$integrable $f$ we write $m f:=\int_{\bbR} f(z) m(dz)$.  The next result is direct consequence of Lemma 1.1 and subsequent remarks in \cite{ko}:
\begin{lemma} There is a $\cP$-valued $\cF^{\tau,t,X}$-optional process $\pi^t(\om, dx)$ such that
\[
\pi^t_s f= \bbE^{\tau,t}[f(Z_s)|\cF^{\tau, t, X}_s]
\]
for all bounded measurable $f$. Moreover, $(\pi_s^t)_{s \in [0,t]}$ has a \cadlag version.
\end{lemma}

Let's recall  the {\em  innovation process}
\[
I_s= X_s - \int_0^s \pi^t_r \kappa_r dr
\]
where $\kappa_r(z):= \frac{q_x (V(r)-r, X_r, z)}{q(V(r)-r, X_r, z)}$. Although it is clear that $I$ depends on $t$, we don't emphasize it in the notation for convenience. Due to Lemma \ref{l:2moment}  $\pi^t_s \kappa_s$ exists for all $s\leq t$.

In order to be able to use the results of \cite{ko} we first need to  establish the Kushner-Stratonovich equation satisfied by $(\pi^t_s)_{s \in [0,t)}$. To this end, let $B(A) $ denote the set of bounded Borel measurable real valued functions on $A$, where $A$ will be alternatively a measurable subset of $\mathbb R_+ ^2$ or a measurable subset of $\mathbb R_+$. Consider  the operator $\cA_0: B([0,t]\times \bbR_+) \mapsto B([0,t]\times \bbR_+)$ defined by
\be \label{e:A0}
\cA_0 \phi (s, x)=\frac{\partial \phi}{\partial s}(s,x) +  \half \sigma^2(s) \frac{\partial^2  \phi}{\partial x^2}(s,x) + \sigma^2(s)  \frac{H_x}{H}(V(t)-V(s),x) \frac{\partial \phi}{\partial x}(s,x),
\ee
with the domain $\mathcal{D}(\cA_0)= C^{\infty}_c([0,t] \times \bbR_+)$, where  $C^{\infty}_c$ is the class of infinitely differentiable functions with compact support.  By Lemma \ref{l:mpz} the martingale problem for $\cA_0$ is well-posed over the time interval $[0, t-\eps]$ for any $\eps>0$. Therefore, it is well-posed on $[0,t)$  and its unique solution is given by $(s, Z_s)_{s\in[0,t)}$ where $Z$ is defined by (\ref{e:z4f}).  Moreover, the Kushner-Stratonovich equation for the conditional distribution of $Z$ is given by the following:
\be \label{ks}
\pi^t_s f = \pi^t_0 f + \int_0^s \pi^t_r( \cA_0 f) dr + \int_0^s \left[ \pi^t_r(\kappa_r f) -\pi^t_r \kappa_r \pi^t_r f \right] dI_r,
\ee
for all $f \in C^{\infty}_c(\bbR_+)$(see Theorem 8.4.3 in \cite{ka} and note that the condition therein is satisfied due to Lemma \ref{l:2moment}). Note that $f$ can be easily made an element of $\mathcal{D}(\cA_0)$ by redefining it as $f \mathbf{n}$ where $\mathbf{n} \in  C^{\infty}_c(\bbR_+)$ is such that  $\mathbf{n}(s)=1$ for all $s \in [0,t)$. Thus, the above expression is rigorous. The following theorem is a corollary to Theorem  4.1 in \cite{ko}.
\begin{theorem} \label{th:munique} Let $m^t$ be an $\cF^{\tau,t,X}$-adapted \cadlag $\cP$-valued process such that
\be \label{ks1}
m^t_s f = \pi^t_0 f + \int_0^s m_r^t( \cA_0 f) dr + \int_0^s \left[ m_r^t(\kappa_r f) -m^t_r \kappa_r m^t_r f \right] dI^m_r,
\ee
for all $f \in C^{\infty}_c(\bbR_+)$, where $I^m_s=X_s-\int_0^s m_r^t\kappa_r\, dr.$ Then, $m^t_s=\pi^t_s$ for all $s<t$, a.s..
\end{theorem}
\begin{proof} Proof follows along the same lines as the proof of Theorem 4.1 in \cite{ko}, even though, differently from \cite{ko}, we allow the drift of $X$ to depend on $s$ and $X_s$, too. This is due to the fact that  \cite{ko} used the assumption that the drift depends only on the signal process, $Z$, in order to ensure that the joint martingale problem $(X,Z)$ is well-posed, i.e. conditions of Proposition 2.2 in \cite{ko} are satisfied.  Note that the relevant martingale problem is well posed in our case  by Proposition \ref{p:mpj}.
 \end{proof}

Now, we can state and prove the following corollary.
\begin{corollary} \label{pdfZ} Let $f \in B(\bbR_+)$. Then,
\[
\pi^t_s f = \int_{\bbR_+} f(z) p(V(t);s,V(s),X_s,z)\, dz,
\]
for $s <t$ where $p$ is as defined in (\ref{e:td-bm}).
\end{corollary}
\begin{proof} Let $\rho(t; s,x,z):=p(V(t);s, V(s), x, z)$. Direct computations lead to
\bea \label{pderho-d}
&&\rho_s + \frac{H_x(V(t)-s, x)}{H(V(t)-s, x)}\rho_x+ \frac{1}{2} \rho_{xx}\\
&=& -\sigma^2(s)\left(\frac{H_x(V(t)-V(s),z)}{H(V(t)-V(s), z)} \rho\right)_z + \frac{1}{2} \sigma^2(s) \rho_{zz}. \nn
\eea
Define $m^t \in \cP$ by
$m^t_s f:=\int_{\bbR_+}f(z)\rho(t;s,X_s, z)dz$. Then, using the above pde and Ito's formula one can directly verify that $m^t$ solves (\ref{ks1}). Finally, Theorem \ref{th:munique} gives the statement of the corollary.

\end{proof}

Now, we have all necessary results to prove Proposition \ref{c_density}.\medskip

{{\sc Proof of Proposition \ref{c_density}.}\hspace{3mm}} Note that as
$X$ is continuous, $\cF^{\tau,t,X}_t=\bigvee_{s < t}
\cF^{\tau,t,X}_s$. Fix $r < t$ and let $ E \in \cF^{\tau,t,X}_r$.  We
will show that for any $f \in C^{\infty}_c(\bbR_+)$
\[
\bbE^{\tau,t}[f(Z_t)|\chf_{E}]=\bbE^{\tau,t}\left[\int_{\bbR_+}f(z)
\frac{q(V(t)-t,X_t,z)}{H(V(t)-t,X_t)}\,dz\, \chf_{E}\right].
\]
Since $Z$ is continuous and $f$ is bounded
we have
\[
\bbE^{\tau,t}[f(Z_t)\chf_{E}]=\lim_{s \uparrow
  t} \bbE^{\tau,t}[ f(Z_s)\chf_{E}].
\]
As $s$ will eventually be larger than $r$, $\chf_E \in \cF^{\tau,t,X}_s$ for large enough $s$ and, then,  Corollary \ref{pdfZ} and another application of the Dominated Convergence Theorem will yield
\bean
\lim_{s \uparrow
  t} \bbE^{\tau,t}[ f(Z_s)\chf_{E}]&=&\lim_{s \uparrow
  t} \bbE^{\tau,t}\left[ \int_{\bbR_+}f(z) p(V(t); V(s)-s, X_s, z)\,dz\,\chf_{E}\right]\\
  &=& \bbE^{\tau,t}\left[\lim_{s \uparrow
  t} \int_{\bbR_+}f(z) p(V(t); V(s)-s, X_s, z)\,dz\,\chf_{E}\right].
\eean
Since $X$ is strictly positive until $\tau$ by Proposition \ref{p:up},
$\min_{s \leq t} X_s >0$. This yields that $\frac{1}{H(V(t)-s, X_s)}$
is bounded ($\om$-by-$\om$) for $s \leq t$. Moreover, $q(V(s)-s, X_s,
\cdot)$ is bounded by $\frac{1}{\sqrt{2 \pi (V(s)-s)}}$. Thus, in view of (\ref{e:td-bm}),
\[
p(V(t); V(s)-s, X_s, z) \leq
\frac{K(\om)}{\sqrt{V(s)-s}}H(V(t)-V(s),z),
\]
where $K$ is a constant.
Since $(V(s)-s)^{-1}$ can be bounded when $s$ is away from $0$, $H$ is
bounded by $1$, and $f$ has a compact support, it follows
from the Dominated Convergence Theorem that
\[
\lim_{s \uparrow
  t} \int_{\bbR_+}f(z) p(V(t); V(s)-s, X_s, z)\,dz
=\int_{\bbR_+}f(z)\frac{q(V(t)-t,X_t,z)}{H(V(t)-t,X_t)}\,dz, \,
\bbP^{\tau, t}-\mbox{a.s..}
\]
This in turn shows,
\[
\bbE^{\tau,t}[f(Z_t)\chf_{E}]=\bbE^{\tau,t}[\lim_{s \uparrow
  t} f(Z_s)\chf_{E}]
=\bbE^{\tau,t}\left[\int_{\bbR_+}f(z) \frac{q(V(t)-t,X_t,z)}{H(V(t)-t,X_t)}\,dz\, \chf_{E}\right].
\]
The claim now follows from (\ref{e:eq}). \hfill \qed

\appendix
\section{Auxiliary results and their proofs}
\subsection{Comparison results}
\begin{lemma} \label{l:compz} Suppose that $d:\bbR_+\times \bbR_+^2\mapsto [0, M]$ for some constant $M>0$ is a measurable function and $Y$ is a strong solution to
\[
Y_t=y+ 2 \int_0^t \sqrt{|Y_s|}dB_s +\int_0^t d(s,Y_s,Z_s)ds
\]
for some $y \geq 0$ upto an explosion time $\tau$. Then, $\bbP[\tau=\infty]=1$ and $\bbP[0 \leq Y_t \leq Y^M_t, \forall t]=1$, where
\[
Y^M_t=y+ 2 \int_0^t \sqrt{|Y^M_s|}dB_s +\int_0^t M ds.
\]
\end{lemma}
\begin{proof}

Let $\tau_n:= \inf\{t>0: |Y_t| \geq n\}$. By Tanaka's formula,
\bean
(Y_{t\wedge \tau_n}-Y^M_{t\wedge \tau_n})^+&=&2 \int_0^{t\wedge \tau_n} (\sqrt{|Y_s|}-\sqrt{Y^M_s})\chf_{[Y_s>Y^M_s]}dB_s \\
&&-\int_0^{t\wedge \tau_n} (M-d(s, Y_s, Z_s))\chf_{[Y_s>Y^M_s]}ds + L^0(Y-Y^M)_{t\wedge \tau_n},\\
Y_{t\wedge \tau_n}^-&=&- 2 \int_0^{t\wedge \tau_n} \sqrt{|Y_s|}\chf_{[Y_s<0]}dB_s \\
&&-\int_0^{t\wedge \tau_n} d(s, Y_s, Z_s)\chf_{[Y_s<0]}ds + L^0(Y)_{t\wedge \tau_n}
\eean
where $L^0(Y-Y^M)$ and $L^0(Y)$ are the local times of $Y-Y^M$ and $Y$ at $0$, respectively. We will first show that $Y$ is nonnegative upto $\tau_n$. Since
\[
\int_0^{t\wedge \tau_n} \chf_{[0< -Y_s <1]}\frac{|Y_s|}{|Y_s|}ds \leq t
\]
and  $\int_0^{\infty} \frac{1}{x} dx=\infty$, it follows from Lemma 3.3 in Chap.~IX of \cite{ry} that $L^0(Y_{t\wedge \tau_n})=0$ for all $t \geq 0$. Thus,
\[
\bbE\left[Y_{t\wedge \tau_n}^-\right]=\bbE\left[-2 \int_0^{t\wedge \tau_n} \sqrt{Y_s}\chf_{[Y_s<0]}dB_s-\int_0^{t\wedge \tau_n} d(s, Y_s, Z_s)\chf_{[Y_s<0]}ds\right] \leq 0,
\]
since the stochastic integral is a martingale having a bounded integrand. Thus, $Y_{t\wedge \tau_n} \geq 0$, a.s. for every $t\geq 0$ and any $n$.

Similarly,
\[
\int_0^{t\wedge \tau_n} \chf_{[0< Y_s-Y^M_s <1]}\frac{(\sqrt{|Y_s|}-\sqrt{Y^M_s})^2}{|Y_s-Y^M_s|}ds=\int_0^{t\wedge \tau_n} \chf_{[0< Y_s-Y^M_s <1]}\frac{(\sqrt{Y_s}-\sqrt{Y^M_s})^2}{|Y_s-Y^M_s|}ds \leq t,
\]
where the first equality is due to the fact that $Y^M\geq 0$ implies $Y_s\geq 0$ on the set $[Y_s-Y^M_s>0],$ and the second inequality follows from the elementary fact that $|\sqrt{x}-\sqrt{y}|\leq \sqrt{|x-y|}$. Thus it follows from Lemma 3.3 in Chap.~IX of \cite{ry} that $L^0(Y_{t\wedge \tau_n}-Y^M_{t\wedge \tau_n})=0$ for all $t \geq 0$ and
\bean
\bbE\left[(Y_{t\wedge \tau_n}-Y^M_{t\wedge \tau_n})^+\right]&=&2 \bbE\left[\int_0^{t\wedge \tau_n} (\sqrt{Y_s}-\sqrt{Y^M_s})\chf_{[Y_s>Y^M_s]}dB_s\right]\\
 &&-\bbE\left[\int_0^{t\wedge \tau_n} (M-d(s, Y_s, Z_s))\chf_{[Y_s>Y^M_s]}ds\right] \leq 0,
\eean
since the stochastic integral $(\int_0^{t\wedge \tau_n} (\sqrt{Y_s}-\sqrt{Y^M_s})\chf_{[Y_s>Y^M_s]}dB_s)_{t \geq 0}$ is a martingale having a bounded integrand. Thus, $Y_{t\wedge \tau_n} \leq Y^M_{t\wedge \tau_n}$, a.s. for every $t\geq 0$ and any $n$. Since  $Y$ and $Y^M$ are continuous upto time $\tau_n$, we have
\[
\bbP[0 \leq Y_{t\wedge \tau_n} \leq Y^M_{t\wedge \tau_n}, \forall t\geq 0]=1.
\]
By taking the limit as $n \rar \infty$, we obtain
\[
\bbP[0 \leq Y_{t\wedge \tau} \leq Y^M_{t\wedge \tau}, \forall t\geq 0]=1.
\]
Since $Y^M$ is non-explosive, this implies that $\tau=\infty$, a.s..
\end{proof}

In view of the above lemma, the hypothesis of the next lemma is not vacuous.
\begin{lemma} \label{l:compstop}  Suppose that $d:\bbR_+\times \bbR_+^2\mapsto [0, M]$ for some constant $M>0$ is a measurable function and $Y$ is the nonnegative strong solution to
\[
Y_t=y+ 2 \int_0^t \sqrt{Y_s}dB_s +\int_0^t d(s,Y_s,Z_s)ds,
\]
for some $y \geq 0$. Moreover, suppose that there exists two stopping times $S \leq T$ such that  $d((t\vee S)\wedge T, Y_{(t\vee S)\wedge T},Z_{(t\vee S)\wedge T}) \in [a,b]\subseteq [0,M]$ for some constants $a$ and $b$.  Then, $\bbP[Y^a_{t\wedge T} \leq Y_{t\wedge T} \leq Y^b_{t\wedge T}, \forall t]=1$, where
\bean
Y^a_t&=&Y_{t \wedge S}+ \int_S^{t \vee S}\left\{2 \sqrt{Y^a_s}dB_s + a ds\right\}\\
Y^b_t&=&Y_{t \wedge S}+ \int_S^{t \vee S}\left\{2 \sqrt{Y^b_s}dB_s + b ds\right\}.
\eean
\end{lemma}
\begin{proof}
Observe that using the similar arguments as in the previous lemma, one obtains that $L^0(Y-Y^a)=L^0(Y-Y^b)=0$. Thus, by Tanaka's formula,
\bean
(Y_{t\wedge T }-Y^b_{t\wedge T})^+&=&2 \int_{t \vee S}^{t \wedge T} (\sqrt{Y_s}-\sqrt{Y^b_s})\chf_{[Y_s>Y^b_s]}dB_s \\
&&-\int_{t \vee S}^{t\wedge T} (b-d(s, Y_s, Z_s))\chf_{[Y_s>Y^b_s]}ds \\
(Y^a_{t\wedge T }-Y_{t\wedge T})^+&=&2 \int_{t \vee S}^{t \wedge T} (\sqrt{Y^a_s}-\sqrt{Y_s})\chf_{[Y^a_s>Y_s]}dB_s \\
&&-\int_{t \vee S}^{t\wedge T} (d(s, Y_s, Z_s)-a)\chf_{[Y^a_s>Y_s]}ds.
\eean
 Observe that the stochastic integrals above are nonnegative local martingales, therefore they are supermartingales. Thus, by taking the expectations we obtain
 \bean
 \bbE\left[(Y_{t\wedge T }-Y^b_{t\wedge T})^+\right]&\leq& 0 \\
 \bbE\left[(Y^a_{t\wedge T }-Y_{t\wedge T})^+\right]&\leq& 0.
 \eean
 Hence, the conclusion follows.
\end{proof}
\subsection{Martingale problems and some $L^2$ estimates}
In the next lemma we show that the martingale problem related to $Z$ as defined in (\ref{e:z4f}) is well posed. Recall that $\cA_0$ is the associated infinitesimal generator defined in (\ref{e:A0}). We will denote the restriction of $\cA_0$ to $B([0,t-\eps] \times \bbR_+)$ by $\cA^{\eps}_0$.

\begin{lemma} \label{l:mpz} Fix $\eps>0$ and let  $\mu \in \cP$. Then, the martingale problem $(\cA_0^{\eps}, \mu)$ is well-posed. Moreover, the SDE (\ref{e:z4f}) has a unique weak solution for any nonnegative initial condition and the solution is strictly positive on $(s, t-\eps]$ for any $s \in [0, t-\eps]$.
\end{lemma}
\begin{proof} Let $ s \in[0, t-\eps]$ and $z \in \bbR_+$. Then, direct calculations yield
\be \label{e:weakz}
dZ_r= \sigma(r)d\beta_r + \sigma^2(r)\left\{\frac{1}{Z_r}-Z_r \eta^t(r,Z_r)\right\}\, dr,  \mbox{ for } r \in [s, t-\eps],
\ee
with $Z_s=z$, where
\be \label{d:etat}
\eta^{t}(r,y):=\frac{\int_{V(t)-V(r)}^{\infty}\frac{1}{\sqrt{2 \pi u^5}}\exp\left(-\frac{y^2}{2 u}\right)du}{\int_{V(t)-V(r)}^{\infty}\frac{1}{\sqrt{2 \pi u^3}}\exp\left(-\frac{y^2}{2 u}\right)du}, \ee
thus, $\eta^{t}(r,y) \in [0, \frac{1}{V(t)-V(t-\eps)}]$ for any $r \in [0, t-\eps]$ and $y \in \bbR_+$.

First, we show the uniqueness of the solutions to the martingale problem. Suppose there exists a weak solution  taking values in $\bbR_+$ to the SDE above. Thus, there exists $(\tilde{Z}, \tilde{\beta})$ on some filtered probability space $(\tilde{\Om}, \tilde{\cF}, (\tilde{\cF}_r)_{r \in [0,t-\eps]}, \tilde{P})$ such that
\[
d\tilde{Z}_r= \sigma(r)d\tilde{\beta}_r + \sigma^2(r)\left\{\frac{1}{\tilde{Z}_r}-\tilde{Z}_r \eta^t(r,\tilde{Z}_r)\right\}\, dr, \qquad \mbox{for } r \in [s, t-\eps],
\]
with $\tilde{Z}_s=z$. Consider $\tilde{R}$ which solves
\be \label{e:3dB}
d\tilde{R}_r= \sigma(r)d\tilde{\beta}_r + \sigma^2(r)\frac{1}{\tilde{R}_r}dr,
\ee
with $\tilde{R}_s=z$. Note that this equation is the SDE for a time-changed 3-dimensional Bessel process  with a deterministic time change and the initial condition $\tilde{R}_s=z$. Therefore, it has a unique strong solution  which is strictly positive on $(s, t-\eps]$ (see 9.~446 in Chap.~XI of \cite{ry}). Then, from Tanaka's formula (see Theorem 1.2 in Chap.~VI of \cite{ry}), since the local time of $\tilde{R}-\tilde{Z}$ at $0$ is identically $0$ (see Corollary 1.9 in Chap.~VI of \cite{ry}), we have
\[
(\tilde{Z}_t-\tilde{R}_t)^+= \int_0^t \chf_{[\tilde{Z}_r>\tilde{R}_r]} \sigma^2(r)\left\{ \frac{1}{\tilde{Z}_r}-\tilde{Z}_r \eta^t(r,\tilde{Z}_r)-\frac{1}{\tilde{R}_r}\right\}dr \leq 0,
\]
where the last inequality is due to  $\eta^{t} \geq 0$, and $\frac{1}{a} < \frac{1}{b}$ whenever $a >b >0$. Thus, $\tilde{Z}_r \leq \tilde{R}_r$ for $r \in [s ,t-\eps]$.

Define $(L_r)_{r \in [0,t-\eps]}$ by $L_0=1$ and
\[
dL_r=  -L_r \tilde{Z}_r\eta^t(r,\tilde{Z}_r)\, d\tilde{\beta}_r.
\]
If $(L_r)_{r \in [0,t-\eps]}$ is a true martingale, then $Q$ on $\tilde{\cF}_{t-\eps}$ defined by
\[
\frac{dQ}{d\tilde{P}}=L_{t-\eps},
\]
 is a probability measure on $\tilde{\cF}_{t-\eps}$ equivalent to $\tilde{P}$. Then, by Girsanov Theorem (see, e.g.,  Theorem 3.5.1 in \cite{ks}) under $Q$
\[
d\tilde{Z}_r= \sigma(r)d\tilde{\beta}^Q_r + \sigma^2(r)\frac{1}{\tilde{Z}_r}\, dr, \qquad \mbox{for } r \in [s, t-\eps],
\]
with $\tilde{Z}_s=z$, where $\tilde{\beta}^Q$ is a $Q$-Brownian motion. This shows that $(\tilde{Z}, \tilde\beta^Q)$ is a weak solution to (\ref{e:3dB}). As (\ref{e:3dB}) has a unique strong solution which is strictly positive on $(s, t-\eps]$, any weak solution to (\ref{e:z4f}) is strictly positive on $(s, t-\eps]$.  Thus, due to Theorem 6.4.2 in \cite{SV}, the martingale problem for $(\delta_z, \cA^{\eps}_0)$ has a unique solution. Note that although the drift coefficient is not bounded, Theorem 6.4.2 in \cite{SV} is still applicable when $L$ is a martingale.

Thus, it remains to show that $L$ is a true martingale when $\tilde{Z}$ is a positive solution to (\ref{e:weakz}). For some $0 \leq t_{n-1} <
t_n \leq t-\eps$  consider
\be \label{novikovz}
\bbE\left[\exp\left(\frac{1}{2}\int_{t_{n-1}}^{t_n }
  (\tilde{Z}_r\eta^t(r,\tilde{Z}_r))^2 dr \right)\right].
\ee
The  expression in (\ref{novikovz})
is bounded by
\[
\bbE\left[\exp\left(\frac{1}{2}\int_{t_{n-1}}^{t_n}
  \tilde{R}^2_r\left(\frac{1}{V(t)-V(t-\eps)}\right)^2 dr \right)\right]
\leq \bbE\left[\exp\left(\frac{1}{2}(\tilde{R}^{\ast}_r)^2
 \frac{t_{n}-{t_{n-1}}}{(V(t)-V(t-\eps))^2} \right)\right]
\]
where $Y^{\ast}_t:=\sup_{s \leq t}|Y_s|$ for any \cadlag process
$Y$.  Recall that $\tilde{R}$ is only a time-changed
Bessel process where the time change is deterministic and, therefore, $\tilde{R}_r^2$ is
the square of the Euclidian norm at time $V(r)$ of a 3-dimensional standard Brownian
motion, starting at $(z,0,0)$ at time $V(s)$. Thus, by using the same arguments as in Proposition \ref{p:euy}, we get that the above expression is
going to be finite if
\[
E^{z}_{V(s)}\left[\exp\left(\frac{1}{2}(\beta^{\ast}_{V(t-\eps)})^2\frac{t_{n}-{t_{n-1}}}{(V(t)-V(t-\eps))^2} \right)\right] <\infty,
\]
where $\beta$ is a standard Brownian motion and  $E^x_s$ is the
expectation with respect to the law of a standard Brownian  motion
starting at $x$ at time $s$. In view of the reflection principle for standard
Brownian motion (see, e.g. Proposition 3.7 in Chap.~3 of \cite{ry}) the above expectation is going to be finite if
\[
\frac{t_{n}-{t_{n-1}}}{(V(t)-V(t-\eps))^2 }< \frac{1}{V(t-\eps)}.
\]

Clearly, we can find
 a finite sequence of real numbers $0=t_0<t_1<\ldots<t_{n(T)}=T$
 that satisfy above. Now, it follows from Corollary 3.5.14 in \cite{ks} that $L$ is
 a martingale.

 In order to show the existence of a nonnegative solution, consider the solution, $\tilde{R}$, to (\ref{e:3dB}), which is a time-changed 3-dimensional Bessel process, thus, nonnegative. Then, define  $(L^{-1}_r)_{r \in [0,t-\eps]}$ by $L^{-1}_0=1$ and
\[
dL^{-1}_r=  L^{-1}_r \tilde{R}_r\eta^t(r,\tilde{R}_r)\, d\tilde{\beta}_r.
\]
Applying the same estimation to $L^{-1}$ as we did for $L$ yields that $L^{-1}$ is a true martingale. Then, $Q$ on $\tilde{\cF}_{t-\eps}$ defined by
\[
\frac{dQ}{d\tilde{P}}=L^{-1}_{t-\eps},
\]
is a probability measure on $\tilde{\cF}_{t-\eps}$ under which $\tilde{R}$ solves
\[
d\tilde{Z}_r= \sigma(r)d\tilde{\beta}^Q_r + \sigma^2(r)\left\{\frac{1}{\tilde{Z}_r}-\tilde{Z}_r \eta^t(r,\tilde{Z}_r)\right\}\, dr, \qquad \mbox{for } r \in [s, t-\eps],
\]
with $\tilde{Z}_s=z$ and $\tilde{\beta}^Q$ is a $Q$-Brownian motion. This means that the nonnegative process $\tilde{R}$ is a weak solution of (\ref{e:weakz}). Therefore, the martingale problem $(\delta_z, \cA^{\eps}_0)$ has a solution by Proposition 5.4.11 and Corollary 5.4.8 in \cite{ks} since $\sigma$ is locally bounded. Thus, the martingale problem $(\delta_z, \cA^{\eps}_0)$ is well-posed for any $z \in \bbR_+$.

The well-posedness of the martingale problem for $(\mu, \cA^{\eps}_0)$ follows from Theorem 21.10 in \cite{ok} since  $P^z$ is the unique solution of the martingale problem for $(\delta_z, \cA^{\eps}_0)$ for any $z \in \bbR_+$.

\end{proof}

We are now ready to show that the joint martingale problem for $(X,Z)$ defined by the operator $\cA:B([0,t) \times \bbR_+^2) \mapsto B([0,t) \times \bbR_+^2)$ which is given by
\bean
\cA \phi (s, x, z)&=&\frac{\partial \phi}{\partial s}(s,x,z) + \half
\frac{\partial^2  \phi}{\partial x^2}(s,x,z) + \half \sigma^2(s)
\frac{\partial^2  \phi}{\partial z^2}(s,x,z) \nn \\
\label{e:A} &&+  \frac{q_x}{q}(V(t)-V(s),x,z) \frac{\partial \phi}{\partial x}(s,x,z)+ \sigma^2(s)  \frac{H_z}{H}(V(t)-V(s),z) \frac{\partial \phi}{\partial z}(s,x,z),
\eean
with the domain $\cD(\cA)=C^{\infty}_c([0,t)\times \bbR_+^2)$.
\begin{proposition} \label{p:mpj} Let  $\mu \in \cP^2$ where $\cP^2$ is the set of probability measures on the Borel sets of $\bbR_+^2$ topologized by weak convergence. Then, the martingale problem $( \mu, \cA)$ is well-posed.
\end{proposition}
\begin{proof}  Clearly, if $(\mu, \cA^{\eps})$ is well-posed for any $\eps>0$, where $\cA^{\eps}$ is the restriction of $\cA$ to  $B([0,t-\eps], \bbR_+)$, then $(\mu, \cA)$ is  well-posed.  As in the proof of Lemma \ref{l:mpz}, the problem of well-posedness of $(\mu, \cA^{\eps})$ can be reduced to that of $(\delta_{x,z}, \cA^{\eps})$ for any fixed  $(x,z)\in \bbR_+^2$ due to Theorem 21.11 in \cite{ok} and Proposition 1.6 in Chap.~III of \cite{ry}. To this end, in view of Proposition 5.4.11 and Corollary 5.4.8 in \cite{ks}, it suffices to show the existence and the uniqueness of weak solutions to the system of SDEs defined by (\ref{e:z4f}) and (\ref{e:x4f}) with the initial condition that $X_s=x$ and $Z_s=z$ for a fixed $s \in [0, t-\eps]$. We will consider the following three cases to finish the proof.

\begin{itemize}
\item[{\bf Case 1:}]{$ x>0,\, z>0$}. In Lemma \ref{l:mpz} we have proved the existence and the uniqueness of a weak solution to the SDE (\ref{e:z4f}) for any initial condition $Z_s=z$ for $s \in [0, t-\eps]$ and $z \geq 0$. Thus, there exists $(\tilde{Z}, \tilde{\beta})$ on some filtered probability space $(\tilde{\Om}, \tilde{\cF}, (\tilde{\cF}_r)_{r \in [0,t-\eps]}, \tilde{P})$ such that $(\tilde{Z},\tilde{\beta})$ solves the SDE (\ref{e:z4f}) with the initial condition $Z_s=z$. Without loss of generality we can assume that the space $(\tilde{\Om}, \tilde{\cF}, (\tilde{\cF}_r)_{r \in [0,t-\eps]}, \tilde{P})$ supports another Brownian motion, $\tilde{B}$, independent of $\tilde{\beta}$. Then, Proposition \ref{p:euy} yields that there exists a unique strong and strictly positive solution to (\ref{e:x4f}) on $(\tilde{\Om}, \tilde{\cF}, (\tilde{\cF}_r)_{r \in [0,t-\eps]}, \tilde{P})$. Indeed, the proof of Proposition \ref{p:euy} would remain the same as long as the initial condition for $Z$ is strictly positive and one observes that although $Z$ is not a Brownian motion, the finiteness of (\ref{e:bnov1}) still follows from (\ref{e:bnov2}) since $Z$ is strictly positive and bounded from above by a time-changed 3-dimensional Bessel process and the time change is given by $V(t)$. This demonstrates that there exists a weak solution to the system of SDEs. Moreover, the solution is unique in law since $X$ is pathwise uniquely determined by $Z$, which is unique in law.
\item[{\bf Case 2:}]{$ x=0,\, z\geq 0$}. We can use the same arguments as in the previous case once we establish Lemma \ref{aux_sde3} over the time interval $[s, t-\eps]$. Note that we only need to show the strict positivity of the solution as the existence  of a nonnegative weak solution follows along the same lines. Consider the sequence of stopping times $(\tau_n)_{n \geq 1}$
    \[
    \tau_n:=\inf\{r \in [s, t -\eps]: U_r =\frac{1}{n}\},
    \]
    where $\inf\emptyset=t-\eps$. On $(\tau_n, t-\eps]$ the solution exists and is strictly positive as in Case 1 since $Z_{\tau_n}>0$ and $U_{\tau_n}= \frac{1}{n}$ when $\tau_n < t-\eps$. Consider $\tau:=\inf_n \tau_n$. If $\tau=s$, we are done. Suppose $\tau>s$ with some positive probability. Then, on this set $U_t =0$ for $t \leq \tau$. However, this contradicts the fact that $U$ solves (\ref{sdeU}) on $[s, t-\eps]$.
\item[{\bf Case 3:}] {$x>0,\,  z=0$}. As in the previous case it only remains to establish the strict positivity of the solution of (\ref{sdeU}), which exists by the same arguments. Again consider the following sequence of stopping times:
 \[
    \tau_n:=\inf\{r \in [s, t -\eps]: Z_r =\frac{1}{n}\},
    \]
where $\inf\emptyset=t-\eps$. That the weak solution to (\ref{sdeU}) is strictly positive on $(\tau_n, t-\eps]$ follows from Case 1 if $X_{\tau_n}>0$, and from Case 2 if $X_{\tau_n}=0$. Since $\inf_n \tau_n =s$ by Lemma \ref{l:mpz}, we have the strict positivity on $[s, t-\eps]$.
\end{itemize}
\end{proof}

\begin{lemma} \label{l:2moment} Let $(Z,X)$ be the unique strong solutions to (\ref{sdeZ}) and (\ref{sdeX}). Then  they solve the martingale problem on the interval $[0,t)$ defined by (\ref{e:z4f}) and (\ref{e:x4f}) with the initial condition $X_0=Z_0=1$. Moreover, under Assumption \ref{a:sigma} we have
\begin{description}
\item[i)] $\bbE\left[\int_0^{t}\chf_{[\tau>s]}\left(\frac{q_x}{q}(V(s)-s, X_s, Z_s)\right)^2\, ds\right] < \infty.$
\item[ii)] $\bbE^{\tau, t}\left[\int_0^{t}\left(\frac{q_x}{q}(V(s)-s, X_s, Z_s)\right)^2\, ds\right] < \infty.$
\item[iii)] $\bbE^{\tau, t}\left[\int_0^{t-\eps}\sigma^2(s)\left|\frac{H_x}{H}(V(t)-V(s),Z_s)\right|\, ds\right]^2 < \infty,$ for any $\eps >0$.
\end{description}

\end{lemma}
\begin{proof} Recall that $\frac{d\bbP^{\tau, t}}{d\bbP_t}=\frac{\chf_{[\tau>t]}}{\bbP[\tau>t]}$ and that $\mathbb E^{\tau,t}$ denotes the expectation operator with respect to $\bbP^{\tau,t}$. Hence, under $\bbP^{\tau, t}$, $(Z,X)$ satisfy (\ref{e:z4f}) and (\ref{e:x4f}) with the initial condition $X_0=Z_0=1$, which implies that they solve the corresponding martingale problem.
\begin{description}
\item[i) \& ii)] Note that
\bean
&& \bbP[\tau>t]\, \bbE^{\tau,
  t}\left[\int_0^{t}\left(\frac{q_x}{q}(V(s)-s, X_s, Z_s)\right)^2\,
  ds\right]\\
&=&\bbE\left[\chf_{[\tau>t]}\int_0^{t}\left(\frac{q_x}{q}(V(s)-s, X_s, Z_s)\right)^2\, ds\right] \\
&\leq& \bbE\left[\int_0^{t}\chf_{[\tau>s]}\left(\frac{q_x}{q}(V(s)-s, X_s, Z_s)\right)^2\, ds\right].
\eean
Thus, it suffices to prove the first assertion since $\bbP[\tau>t]>0$ for all $t \geq 0$. Recall from (\ref{qx/q}) that
\[
\frac{q_x(t,x,z)}{q(t,x,z)}=\frac{z-x}{t} +
\frac{\exp\left(-\frac{2xz}{t}\right)}{1-\exp\left(-\frac{2xz}{t}\right)}\frac{2z}{t}=\frac{z-x}{t}+f\left(\frac{2xz}{t}\right)\frac{1}{x},
\]
where $f(y)=\frac{e^{-y}}{1-e^{-y}}y$ is bounded by $1$ on $[0,\infty)$.  As $\int_0^{t}\frac{1}{(V(s)-s)^2}ds<\infty$ and $\sup_{s \in [0, t]} \bbE[Z_s^2]\leq  V(t) +1$, the result will follow once we obtain
\begin{enumerate}
\item $\sup_{s \in [0, t]} \bbE[X_s^2\chf_{[\tau>s]}] <\infty$, and
\item $\bbE\left(\int_0^{t}\chf_{[\tau>s]}\frac{1}{X_s^2}ds\right) < \infty$,
\end{enumerate}
demonstrated below.
\begin{enumerate}
\item By Ito formula,
\be \label{e:xt2I}
\chf_{[\tau>t]}X_t^2=\chf_{[\tau>t]}\left(1+ 2 \int_0^t X_s dB_s +2 \int_0^t\left\{\frac{Z_s X_s-X_s^2}{V(s)-s} + f\left(\frac{2 Z_s X_s}{V(s)-s}\right) +\half \right\}\,ds\right).
\ee

Observe that the elementary inequality $2 a b \leq a^2 +b^2$ implies
\bean
2 \chf_{[\tau>t]} \int_0^t X_s dB_s & \leq & 1 +  \left(\chf_{[\tau>t]}\int_0^t X_s dB_s\right)^2 \leq 1 + \left(\int_0^{\tau \wedge t}X_s dB_s\right)^2, \mbox{ and} \\
2 \int_0^t\frac{Z_s X_s-X_s^2}{V(s)-s}\, ds &\leq & \int_0^t \frac{Z_s^2-X_s^2}{V(s)-s}\, ds \leq \int_0^t \frac{Z_s^2}{V(s)-s}\, ds.
\eean
As $f$ is bounded by $1$, using the above inequalities and taking expectations of both sides of (\ref{e:xt2I}) yield
\bean
\bbE[\chf_{[\tau>t]}X_t^2] &\leq& 2 + \bbE\left(\int_0^t \chf_{[\tau > s]}X_s dB_s\right)^2+ \int_0^t \frac{\bbE[Z^2_s]}{V(s)-s}ds + 3t\\
&\leq& 2 + 3t + (V(t)+1)\int_0^t\frac{1}{V(s)-s}\,ds + \int_0^t \bbE\left(\chf_{[\tau>s]}X_s^2\right)ds.
\eean
The last inequality obviously holds when $\int_0^t \bbE\left(\chf_{[\tau>s]}X_s^2\right)ds=\infty$, otherwise, it is a consequence of Ito isometry. Let $T>0$ be a constant, then for all $t \in [0,T]$ it follows from Gronwall's inequality that
\[
\bbE[\chf_{[\tau>t]}X_t^2] \leq \left( 2 + 3T + (V(T)+1)\int_0^T\frac{1}{V(s)-s}\,ds\right)e^T.
\]
\item In view of Proposition \ref{p:compare} we have $\chf_{[\tau>s]}\frac{1}{X_s^2} \leq \frac{1}{R_s^2}$ where $R$ is the unique strong solution of (\ref{e:tcb}). Thus, it is enough to show that
    $\int_0^t \bbE\left[\frac{1}{R_s^2} \right] ds <\infty $. Recall from Proposition \ref{p:tcb} that the law of $R_s$ is that of $\lambda_s \rho_{\Lambda_s}$ where $\rho$ is a 3-dimensional Bessel process starting at 1 and
\bean
\lambda_t&=&\exp\left(-\int_0^t \frac{1}{V(s)-s}\, ds\right), \\
\Lambda_t&=&\int_0^t \frac{1}{\lambda^2_s}\,ds.
\eean
Therefore, using the explicit form of the probability density of 3-dimensional Bessel process (see Proposition 3.1 in Chap.~VI of \cite{ry}) one has
\bean
\int_0^t \bbE\left[\frac{1}{R_s^2}\right] ds &\leq& \int_0^t \bbE\left[\frac{1}{R_s^2} \chf_{[R_s \leq \sqrt[3]{\Lambda_s}]}+ \Lambda_s^{-\frac{2}{3}}\right]ds \\
&\leq&\int_0^t \lambda^{-2}_s \int_0^{\sqrt[3]{\Lambda_s}\lambda^{-1}_s}\frac{1}{y}q(\Lambda_s,1,y)\,dy\, ds + 3 \sqrt[3]{\Lambda_t} \\
&=&\int_0^t \lambda^{-2}_s \int_0^{\sqrt[3]{\Lambda_s}\lambda^{-1}_s}q_y(\Lambda_s,1,y^{\ast})\,dy\, ds + 3 \sqrt[3]{\Lambda_t}
\eean
where the last equality is due to the Mean Value Theorem and $y^{\ast} \in [0,y]$. It follows from direct computations that $|q_y(t,1,y)| \leq \sqrt{\frac{2}{\pi e}}\frac{1}{t}$ for all $y \in \bbR$ and $t \in \bbR_+$. Therefore, we have
\bean
\int_0^t \bbE\left[\frac{1}{R_s^2}\right] ds &\leq& \sqrt{\frac{2}{\pi e}} \int_0^t \lambda^{-2}_s \int_0^{\sqrt[3]{\Lambda_s}\lambda^{-1}_s}\frac{1}{\Lambda_s}\,dy\, ds+ 3 \sqrt[3]{\Lambda_t}\\
&=&\sqrt{\frac{2}{\pi e}} \int_0^t \lambda^{-3}_s \Lambda^{-\frac{2}{3}}_s\,ds + 3 \sqrt[3]{\Lambda_t}\\
&\leq&3 \left(\sqrt{\frac{2}{\pi e}} \lambda^{-1}_t +1\right)\sqrt[3]{\Lambda_t} .
\eean
\end{enumerate}
\item[iii)] Recall that
\[
\frac{H_x}{H}(V(t)-V(s), Z_s)= \frac{1}{Z_s}- Z_s \eta^t(s,Z_s),
\]
where $\eta^t$ is as defined in (\ref{d:etat}). Fix an $\eps >0$. Then,
 \[
 \int_0^{t-\eps} \sigma^2(s)\left|\frac{H_x}{H}(V(t)-V(s), Z_s)\right|ds= \int_0^{V(t-\eps)}\left|\frac{H_x}{H}(V(t)-s, Z_{V^{-1}(s)})\right|ds. \]
Consider the process $S_r:=Z_{V^{-1}(r)}$ for $r \in [0, V(t))$. Then,
\be \label{e:tch}  \begin{split}
& \bbE^{\tau,t}\left[\int_0^{t-\eps}
  \sigma^2(s)\left|\frac{H_x}{H}(V(t)-V(s), Z_s)\right|ds\right]^2 \\
&=\bbE^{\tau,t}\left[\int_0^{V(t-\eps)}\left|\frac{1}{S_s}-S_s
    \eta^t(V^{-1}(s), S_s)\right|ds\right]^2  \\
&\leq 2
\left(\bbE^{\tau,t}\left[\int_0^{V(t-\eps)}\frac{1}{S_s}ds\right]^2 +
  \frac{V(t-\eps)}{(V(t)-V(t-\eps))^2}
  \int_0^{V(t-\eps)}\bbE^{\tau,t}[S^2_s]ds\right).
\end{split}
\ee
Moreover, under $\bbP^{\tau,t}$
\[
dS^2_s=(3 -2 S_s^2 \eta^t(V^{-1}(s), S_s))ds + 2 S_s dW^t_s
\]
for all $s < V(t)$ for the Brownian motion $W^t$ defined by $W^t_s:=\int_0^{V^{-1}(s)}\sigma^2(r) d\beta^t_r.$ Thus,
\[
\bbE^{\tau,t}[S^2_s] \leq 3 s + 1+ \int_0^s \bbE^{\tau,t}[S^2_r] dr.
\]
Hence, by Gronwall's inequality, we have $\bbE^{\tau,t}[S^2_s] \leq (3 s +1)e^{s}$. In view of (\ref{e:tch}) to demonstrate iii) it suffices to show that
\[
\bbE^{\tau,t}\left[\int_0^{V(t-\eps)}\frac{1}{S_s}ds\right]^2 < \infty.
\]
However,
\[
\left(\int_0^{V(t-\eps)}\frac{1}{S_s}ds\right)^2 = \left(S_{V(t-\eps)} -S_0 - W^t_{V(t-\eps)} + \int_0^{V(t-\eps)} \eta^t(V^{-1}(s), S_s) S_s ds\right)^2,
\]
which obviously has a finite expectation due to earlier results.
\end{description}
\end{proof}
\end{document}